\documentclass[12pt]{article}
\usepackage{amsmath,amssymb,amsfonts,amsthm,bbm,xcolor}
\usepackage[round]{natbib}
\usepackage[margin=1in]{geometry}
\usepackage{graphicx}
\RequirePackage[colorlinks,citecolor=blue,urlcolor=blue,breaklinks]{hyperref}

\newtheorem{theorem}{Theorem}
\newtheorem{lemma}[theorem]{Lemma}
\newtheorem{proposition}[theorem]{Proposition}
\newtheorem{corollary}[theorem]{Corollary}
\newcommand{\R}{\mathbb{R}}
\newcommand{\eps}{\epsilon}
\newcommand{\EE}[1]{\mathbb{E}\left[{#1}\right]}

\newcommand{\Ep}[2]{\mathbb{E}_{{#1}}\left[{#2}\right]}

\newcommand{\PP}[1]{\mathbb{P}\left\{{#1}\right\}}

\newcommand{\Pp}[2]{\mathbb{P}_{{#1}}\!\left\{{#2}\right\}}

\newcommand{\One}[1]{{\mathbf{1}}\left\{{#1}\right\}}
\newcommand{\one}[1]{{\mathbf{1}}_{{#1}}}
\newcommand{\iidsim}{\stackrel{\textnormal{iid}}{\sim}}
\newcommand{\norm}[1]{\left\|{#1}\right\|}

\newcommand{\Pcal}{\mathcal{P}}
\newcommand{\Fcal}{\mathcal{F}}
\newcommand{\dkl}{\textnormal{d}_{\textnormal{KL}}}
\newcommand{\dw}{\textnormal{d}_{\textnormal{W}}}
\newcommand{\dH}{\textnormal{d}_{\textnormal{H}}}
\newcommand{\bA}{\mathbf{A}}
\newcommand{\Pemp}{\widehat{P}_n}

\DeclareMathOperator*{\argmax}{argmax}

\newcommand{\Leb}{\mathrm{Leb}}

\title{Local continuity of log-concave projection, with applications to estimation under model misspecification}

\author{Rina Foygel Barber\thanks{Department of Statistics, University of Chicago} \ and Richard J. Samworth\thanks{Statistical Laboratory, University of Cambridge}}

\begin{document}
\maketitle
\begin{abstract}
The log-concave projection is an operator that maps a $d$-dimensional distribution $P$ to 
an approximating log-concave density. It is known that, with suitable metrics on the underlying spaces, this projection is continuous, but not uniformly continuous.
In this work we prove a local uniform continuity result for  log-concave projection---in particular, establishing that this map
is locally H{\"o}lder-(1/4) continuous. A matching lower bound verifies that this exponent cannot be improved.
We also examine the implications of this continuity result for the empirical setting---given a sample drawn from a distribution $P$,
we bound the squared Hellinger distance between
the log-concave projection of the empirical distribution of the sample, and the log-concave projection of $P$. In particular, 
this yields interesting statistical results for the misspecified setting, where $P$ is not itself log-concave.
\end{abstract}

\section{Introduction}\label{sec:intro}

In nonparametric statistics and inference, many problems are formulated in terms of shape constraints.
Examples include isotonic regression and convex regression (for supervised learning problems, placing constraints on the shape of the regression
function relating the response to the covariates), and monotone or log-concave density estimation (for unsupervised learning problems, placing
constraints on a distribution that is the target we wish to estimate). 

Among these examples,
log-concave density estimation is especially challenging in that it cannot be formulated as an $L_2$-projection onto a convex constraint set.
Remarkably, projection onto the space of log-concave densities can still be uniquely defined, but unlike a convex
projection, this operation is not uniformly continuous \citep{dumbgen2011approximation} and its mathematical and statistical properties
are therefore difficult to analyze.
In this work, we examine the continuity properties of log-concave projection more closely to establish locally uniform convergence,
and study the statistical implications of these results.

\subsection{Background}
We begin by establishing some notation used throughout the paper, and then give background on  log-concave projection
and its known properties.

\subsubsection{Notation}
Throughout the paper, $\norm{\cdot}$ denotes the usual Euclidean norm. 
For a distribution $P$, we write $\Ep{P}{\cdot}$ and $\Pp{P}{\cdot}$ to denote expectation or probability taken with respect to a random variable or vector $X$ drawn
from distribution $P$, and $\mu_P := \Ep{P}{X}$ denotes its mean. We will analogously write $\Ep{f}{\cdot}$, $\Pp{f}{\cdot}$, and $\mu_f$ for a density $f$.  We say a distribution, density, or random vector is isotropic if it has zero mean and identity covariance matrix.  Given $x\in\R^d$ and $r > 0$, we write $\mathbb{B}_d(x,r) := \{y \in \mathbb{R}^d:\|y-x\| \leq r\}$ for the closed Euclidean ball of radius $r$ centered at $x$, $\mathbb{B}_d(r)=\mathbb{B}_d(0,r)$  for the closed Euclidean ball of radius $r$ centered at 
zero, and $\mathbb{S}_{d-1}(r) := \{y\in\mathbb{R}^d : \|y\|=r\}$ for the sphere of radius $r$ centered at zero. For the unit ball and unit sphere we write $\mathbb{B}_d=\mathbb{B}_d(1)$ and $\mathbb{S}_{d-1}=\mathbb{S}_{d-1}(1)$. For $x\in\R$, $(x)_+$ denotes $\max\{x,0\}$, and $(x)_-$ denotes $\max\{-x,0\}$.  For independent observations $X_1,\dots,X_n\in\R^d$, we will write $\Pemp$ to denote the empirical distribution.  We write $\Leb_d$ for Lebesgue measure on~$\mathbb{R}^d$.

The $L_1$-Wasserstein distance $\dw$ is
defined for two distributions $P,Q$ on $\R^d$ as 
\[\dw(P,Q) := \inf\left\{\Ep{\tilde{P}}{\norm{X-Y}} \ : \ \begin{tabular}{c}\textnormal{Distributions $\tilde{P}$ on $(X,Y)\in \R^d\times\R^d$}\\
\textnormal{such that marginally $X\sim P$ and $Y\sim Q$}\end{tabular}\right\} \in[0,+\infty].\]
For any distributions $P,Q$ on $\R^d$, this infimum is attained for some coupling $\tilde{P}$ \citep[][Theorem~4.1]{villani2008optimal}.  We will also use the Hellinger distance $\dH$, defined for densities $f,g$ on $\R^d$ as
\[\dH^2(f,g) := \int_{\R^d} \left(\sqrt{f(x)} - \sqrt{g(x)}\right)^2\;\mathsf{d}x.\]
The Hellinger distance is known to satisfy $0\leq \dH^2(f,g)\leq \min\{2, \dkl(f||g)\}$ for any densities $f,g$, where $\dkl(f||g) := \Ep{f}{\log \big(f(X)/g(X)\big)}$ is the Kullback--Leibler divergence. Both $\dw$ and $\dH$ satisfy the triangle inequality, while $\dkl$ does not.

\subsubsection{The log-concave projection}
\label{SubSec:LCP}

For any $d\in \mathbb{N}$, let $\Pcal_d$ denote the set of probability distributions $P$ on $\R^d$ satisfying $\Ep{P}{\norm{X}} < \infty$ and $\Pp{P}{X\in H} < 1$ for every hyperplane $H \subseteq \R^d$, that is, $P$ does not place all its mass in any hyperplane.  Further, let $\Fcal_d$ denote the set of all upper semi-continuous, log-concave densities on $\R^d$.  Then, by \citet[][Theorem~2.2]{dumbgen2011approximation}, there exists a well-defined projection $\psi^*:\Pcal_d \rightarrow \Fcal_d$, given by
\[
  \psi^*(P) := \argmax_{f \in \Fcal_d} \Ep{P}{\log f(X)}.
\]
When $P \in \Pcal_d$ has a (Lebesgue) density~$f_P$ satisfying $\Ep{f_P}{\bigl|\log f_P(X)\bigr|} < \infty$, we can see that $\psi^*(P)$ is the (unique) minimizer over $f\in\Fcal_d$ of the Kullback--Leibler divergence from $f_P$ to $f$---since the KL divergence acts as a sort
of distance, we can think of $f=\psi^*(P)$ as the ``closest'' log-concave density to $f_P$, which 
explains the use of the terminology `projection' to describe this map. In particular, if $f_P$ itself is log-concave, then $\psi^*(P)=f_P$.

To see the gain of defining $\psi^*$ more broadly (i.e., on all distributions $P\in\Pcal_d$, rather than only on distributions with
densities), consider the empirical setting, where $\Pemp$ is the empirical distribution of a sample. 
Then the result of \citet[][Theorem~2.2]{dumbgen2011approximation}  tells us that, provided the convex hull of the data is $d$-dimensional, there exists a unique log-concave maximum likelihood estimator.  We can therefore carry out log-concave density estimation via maximum likelihood in much the same way as if the class $\Fcal_d$ were a standard parametric model.  
To understand the estimation properties of this procedure, 
suppose we metrise $\Pcal_d$ with the $L_1$-Wasserstein distance $\dw$,
and metrise $\Fcal_d$ with the Hellinger distance $\dH$.
Then, by \citet[][Theorem~2.15]{dumbgen2011approximation}, the map $\psi^*$ is continuous. 
For the empirical distribution $\Pemp$ obtained by drawing a sample $X_1,\dots,X_n\iidsim P$, we therefore have
\[
  \dH\bigl(\psi^*(\Pemp),\psi^*(P)\bigr) \stackrel{\mathrm{a.s.}}{\rightarrow} 0.
\]
(This follows from the above continuity result because, 
 by Varadarajan's theorem \citep[][Theorem~11.4.1]{dudley2002real} and the strong law of large numbers, 
 it holds that $d_{\mathrm{W}}(\Pemp,P) \stackrel{\mathrm{a.s.}}{\rightarrow} 0$.)
Thus, if $P \in \Pcal_d$ has a log-concave density, then the log-concave maximum likelihood estimator is strongly consistent---and moreover, even if the log-concavity is misspecified, then the estimator $\psi^*(\Pemp)$ still converges to the log-concave projection $\psi^*(P)$ of $P$.  In this sense, then, the log-concave maximum likelihood estimator converges to the closest element of $\Fcal_d$ to $P$, so can be regarded as robust to misspecification.

Despite these positive results establishing continuity and consistency of $\psi^*$, however, the situation appears much less promising when it comes to  obtaining rates of convergence (e.g., via a Lipschitz-type property of the map).  Indeed, 
we cannot hope for Lipschitz continuity of this map, since 
the review article by \citet{samworth2018recent} gives the following example to show that $\psi^*$ is not even uniformly continuous: let $P^{(n)} = \textnormal{Unif}[-1/n,1/n]$ and $Q^{(n)} = \textnormal{Unif}[-1/n^2,1/n^2]$.  Then $d_{\mathrm{W}}(P^{(n)},Q^{(n)}) \rightarrow 0$, but since $P^{(n)}$ and $Q^{(n)}$ have log-concave densities $f^{(n)} := \frac{n}{2}\one{[-1/n,1/n]}$ and $g^{(n)} := \frac{n^2}{2}\one{[-1/n^2,1/n^2]}$ respectively, we deduce that 
\begin{equation}
\label{Eq:UCCounter}
\dH\bigl(\psi^*(P^{(n)}),\psi^*(Q^{(n)})\bigr)=\dH\bigl(f^{(n)},g^{(n)}\bigr) \nrightarrow 0.
\end{equation}

\paragraph{Summary of contributions} While we have seen that  log-concave projection
does not satisfy uniform continuity, a natural question is whether it may be possible to place further restrictions on the class $\Pcal_d$ to obtain 
a result of this type.  Moreover, from the statistical point of view, we would like
to find a uniform rate of convergence for $\dH\bigl(\psi^*(\Pemp),\psi^*(P)\bigr)$, where $\Pemp$ is the empirical distribution of a sample of size $n$ drawn from $P \in \Pcal_d$, which again might require stronger assumptions than simply $P\in\Pcal_d$.

The first main result of this paper (Theorem~\ref{thm:main}) reveals that the metric space map $\psi^*:(\Pcal_d,d_{\mathrm{W}}) \rightarrow (\Fcal_d,\dH)$ is \emph{locally} H\"older-(1/4) continuous, which establishes a precise understanding of the continuity properties of  log-concave projection.
Theorem~\ref{thm:lowerbd} establishes a matching lower bound, revealing that the exponent $1/4$ cannot be improved.
Next, we specialise to the empirical setting, proving a bound on $\Ep{P}{\dH^2\bigl(\psi^*(\Pemp),\psi^*(P)\bigr)}$ in  Theorem~\ref{thm:main_empirical}.
For $d \geq 2$, this result is a straightforward consequence
of combining our main result in Theorem~\ref{thm:main} with the recent work of \citet{lei2020convergence}, which bounds $\dw(\Pemp,P)$ in expectation, while the case $d=1$ requires a completely different approach.  To the best of our knowledge, this work provides the first understanding of the range of possible rates of convergence of the log-concave maximum likelihood estimator in the misspecified setting.

\subsection{Outline of paper} The remainder of the paper is organized as follows. 
In Section~\ref{sec:mainresults} we present our main results, establishing
the local H{\"o}lder continuity of  log-concave projection, and examining the empirical setting, as described above.
We review prior work on  log-concave projection and related problems in Section~\ref{Sec:PriorWork}.
The proofs of our main results are presented in Section~\ref{sec:proofs}, with technical details deferred to Appendix~\ref{app:proofs}.

\section{Main results}\label{sec:mainresults}
As mentioned in Section~\ref{sec:intro}, \citet[Theorem 2.15]{dumbgen2011approximation} show that the log-concave projection operator 
$\psi^*$
satisfies continuity with respect to appropriate metrics:
\begin{equation}\label{eqn:DSS_2.15}
\textnormal{
The log-concave projection $\psi^*:(\Pcal_d,\dw)\rightarrow(\Fcal_d,\dH)$
is a continuous map.}
\end{equation}
Our main results examine the continuity of the log-concave projection operator $\psi^*$ more closely,
and establish local uniform continuity results.
To do this, we first introduce, for any distribution $P$ on $\R^d$ with $\Ep{P}{\norm{X}}<\infty$, the quantity
\[
  \eps_P :=  \inf_{u\in\mathbb{S}_{d-1}} \Ep{P}{ \big|u^\top(X-\mu_P)\big|}.
\]
The quantity $\eps_P$ can be thought of as a robust analogue of the minimum eigenvalue of the covariance matrix of the distribution $P$ (note that its definition does not require $P$ to have a finite second moment).  
We can also interpret $\eps_P$ as measuring the extent to which $P$ avoids placing all its mass on a single hyperplane.

First, we verify that $\eps_P$ is positive for all $P\in\Pcal_d$, and is Lipschitz with respect to the Wasserstein distance.
\begin{proposition}\label{prop:eps_P}
We have $\eps_P>0$ for any $P\in\Pcal_d$. Furthermore, $|\eps_P-\eps_Q|\leq2\dw(P,Q)$ for any distributions $P,Q$ on $\R^d$ with $\Ep{P}{\norm{X}},\Ep{Q}{\norm{X}}<\infty$.
\end{proposition}
We now present our first main result, which shows that $\eps_P$ allows for a more detailed analysis of the continuity of the map $\psi^*$.  
\begin{theorem}\label{thm:main}
For any $d \geq 1$ and $P,Q\in\Pcal_d$,
\[\dH\bigl(\psi^*(P),\psi^*(Q)\bigr)\leq   C_d\cdot \left[\frac{\dw(P,Q)}{\max\{\eps_P,\eps_Q\}}\right]^{1/4},\]
where $C_d > 0$ depends only on $d$.
\end{theorem}
\noindent This upper bound immediately implies the continuity result~\eqref{eqn:DSS_2.15}, but more importantly, to the best of our knowledge, this is the first general, quantitative statement about the local continuity of log-concave projection.   Another consequence is that, when $d=1$, the uniform continuity counterexample in~\eqref{Eq:UCCounter} is in some sense canonical: if $(P^{(n)})$ and $(Q^{(n)})$ are sequences in $\Pcal_1$ satisfying $\dw(P^{(n)},Q^{(n)}) \rightarrow 0$ and $\liminf_{n \rightarrow \infty} \max\{\eps_{P^{(n)}},\eps_{Q^{(n)}}\} > 0$, then $\dH\bigl(\psi^*(P^{(n)}),\psi^*(Q^{(n)})\bigr) \rightarrow 0$.

\subsection{Extension to affine transformations}
By \citet[][Remark~2.4]{dumbgen2011approximation}, log-concave projection commutes with affine transformations; i.e., if $\psi^*(P)=f$ then $\psi^*(\bA\circ P) = \bA\circ f$ for any invertible
matrix $\bA$,
where $\bA\circ P$ denotes the distribution obtained by drawing $X\sim P$ and returning $\bA X$,
and similarly $\bA\circ f$ denotes the density of the random variable obtained by drawing $X$ according to density $f$ and returning $\bA X$.

Turning to the terms appearing in Theorem~\ref{thm:main}, 
the Hellinger distance is invariant to affine transformations, but the terms on the right-hand side---namely, $\dw(P,Q)$
and $\max\{\eps_P,\eps_Q\}$---are not. By considering affine transformations,
we obtain the following corollary to Theorem~\ref{thm:main}, which we state without further proof:
\begin{corollary}\label{cor:affine}
For any $d \geq 1$ and $P,Q\in\Pcal_d$,
\[\dH\bigl(\psi^*(P),\psi^*(Q)\bigr)\leq C_d\cdot  \inf_{\bA\in\R^{d\times d}, \textnormal{rank}(\bA)=d}\left[\frac{\dw(\bA\circ P,\bA\circ Q)}{\max\{\eps_{\bA\circ P},\eps_{\bA\circ Q}\}}\right]^{1/4},\]
where $C_d > 0$ depends only on $d$.
\end{corollary}

\subsection{A matching lower bound}
To see that our main result in Theorem~\ref{thm:main} is optimal in terms of its dependence on the Wasserstein distance $\dw(P,Q)$ and on
the terms $\eps_P,\eps_Q$, we now construct 
an explicit example to provide a matching lower bound. 
\begin{theorem}\label{thm:lowerbd}
Fix any $d \geq 1$, $\eps>0$, and $\delta>0$. Then there exist distributions $P,Q\in\Pcal_d$ with $\eps_P,\eps_Q\geq \eps$ and $\dw(P,Q)\leq \delta$,
such that
\[\dH\bigl(\psi^*(P),\psi^*(Q)\bigr)\geq c_d\cdot \min\bigl\{1, \big(\delta/\eps\big)^{1/4}\bigr\},\]
where $c_d>0$ depends only on dimension $d$. 
\end{theorem}
The theorem will be proved using the following construction:
Let $P\in\Pcal_d$ be the uniform distribution on the sphere $\mathbb{S}_{d-1}(\rho)$, where $\rho\propto \eps$,
and let $Q\in\Pcal_d$ be the mixture distribution that, with probability $\beta\propto \delta/\eps$, draws uniformly from $\mathbb{S}_{d-1}(2\rho)$, and with probability $1-\beta$ draws uniformly from $\mathbb{S}_{d-1}(\rho)$.  Then $\dw(P,Q) = \rho\beta\propto \delta$, and we will see that $\dH\bigl(\psi^*(P),\psi^*(Q)\bigr) \propto \big(\delta/\eps\big)^{1/4}$, as desired.

\subsection{Bounds for empirical processes}

Now let $X_1,\dots,X_n\iidsim P \in \mathcal{P}_d$, with corresponding empirical distribution function $\Pemp$.  Under an additional moment assumption on $P$, we consider the problem of bounding $\dH^2\bigl(\psi^*(\Pemp),\psi^*(P)\bigr)$. However, to be fully precise, we need
to consider the possibility that $\psi^*(\Pemp)$ may not be defined---specifically,
if $P$ places positive probability on some hyperplane $H\subseteq\R^d$, then it is possible that the empirical distribution $\Pemp$ 
may place all its mass on this hyperplane, in which case we have $\Pemp\not\in\Pcal_d$ and $\psi^*(\Pemp)$ is not defined.
In a slight abuse of notation, for such a case
we will interpret $\dH^2\bigl(\psi^*(\Pemp),\psi^*(P)\bigr)$ as the maximum possible squared Hellinger distance (i.e., $2$).

\begin{theorem}\label{thm:main_empirical}
Fix any $P\in\Pcal_d$, and assume that
\[\Ep{P}{\norm{X}^q}^{1/q}\leq M_q\]
for some $q > 1$.  Let $X_1,\dots,X_n\iidsim P$ for some $n\geq 2$, and let $\Pemp$ denote the corresponding empirical distribution.
Then
\[\EE{\dH^2\bigl(\psi^*(\Pemp),\psi^*(P)\bigr)} \leq C_{d,q}\cdot \sqrt{\frac{M_q}{\eps_P}} \cdot \frac{\log^{3/2} n}{n^{\min\left\{\frac{1}{2d},\frac{1}{2}-\frac{1}{2q}\right\}}},\]
where $C_{d,q} > 0$ depends only on $d$ and $q$.
\end{theorem}
\begin{proof}[Proof of Theorem~\ref{thm:main_empirical}]
First we consider the case $d\geq 2$. The result will follow by combining the bound~\eqref{eqn:dH_dW_may_be_degenerate}, obtained from Theorem~\ref{thm:main}, together
with a bound on the expected Wasserstein distance between $\Pemp$ and $P$ \citep{lei2020convergence}.
Specifically, \citet[Theorem 3.1]{lei2020convergence} establishes that\footnote{In fact, \citet[Theorem 3.1]{lei2020convergence} shows that the $\log^2 n$ term may be reduced to $(\log n)\one{\{d=1,q=2\}} + (\log n)\one{\{d=2,q>2\}} + (\log^2 n)\one{\{d=2,q=2\}} + (\log n)\one{\{d\geq 3,q = d/(d-1)\}}$.  Since poly-logarithmic factors are not our primary concern in this work, however, we will present simpler bounds based on~\eqref{eqn:dw_Lei}.}
\begin{equation}\label{eqn:dw_Lei}\EE{\dw(\Pemp,P)}\leq \tilde{C}_q M_q\cdot \frac{\log^2 n}{ n^{\min\left\{\frac{1}{2},\frac{1}{d},1-\frac{1}{q}\right\}}}\end{equation}
for some $\tilde{C}_q > 0$ depending only on $q$.  
Furthermore,  on the event that $\Pemp\in\Pcal_d$ (i.e., $\Pemp$ does not place
all its mass in any hyperplane), then by applying Theorem~\ref{thm:main} with $Q = \Pemp$ we have
\[\dH^2\bigl(\psi^*(\Pemp),\psi^*(P)\bigr) \leq C_d^2\cdot \frac{\dw^{1/2}(\Pemp,P)}{\max\{\eps_P^{1/2},\eps_{\Pemp}^{1/2}\}}.\]
If instead $\Pemp$ does place all its mass in a hyperplane and so $\psi^*(\Pemp)$ is undefined, 
then in this case we have  $\eps_{\Pemp}=0$, and so by Proposition~\ref{prop:eps_P},
$2\dw(\Pemp,P) \geq |\eps_{\Pemp}-\eps_P| = \eps_P$.
Recalling from above that we interpret $\dH^2\bigl(\psi^*(\Pemp),\psi^*(P)\bigr)$ as equal to $2$ in the case where $\Pemp\not\in \Pcal_d$,
we can see that in either case, it holds that
\begin{equation}\label{eqn:dH_dW_may_be_degenerate}\dH^2\bigl(\psi^*(\Pemp),\psi^*(P)\bigr) \leq \max\left\{C_d^2,\sqrt{8}\right\}\cdot \frac{\dw^{1/2}(\Pemp,P)}{\eps_P^{1/2}}.\end{equation}

Now, taking the expected value and combining the bounds~\eqref{eqn:dw_Lei} and~\eqref{eqn:dH_dW_may_be_degenerate}, we obtain
\begin{multline}\EE{\dH^2\bigl(\psi^*(\Pemp),\psi^*(P)\bigr)}\leq   \EE{\max\left\{C_d^2,\sqrt{8}\right\}\cdot \frac{\dw^{1/2}(\Pemp,P)}{\eps_P^{1/2}}}\\ \leq \max\left\{C_d^2,\sqrt{8}\right\}\cdot \left[\frac{\EE{\dw(\Pemp,P)}}{\eps_P}\right]^{1/2} \nonumber \leq \max\left\{C_d^2,\sqrt{8}\right\}\sqrt{\tilde{C}_q} \cdot \sqrt{\frac{M_q}{\eps_P}} \cdot \frac{\log n}{n^{\min\left\{\frac{1}{4},\frac{1}{2d},\frac{1}{2}-\frac{1}{2q}\right\}}}.\end{multline}
Choosing $C_{d,q} = \max\left\{C_d^2,\sqrt{8}\right\}\cdot \sqrt{\tilde{C}_q}$, this proves the desired result for the case $d\geq 2$.

For the case $d=1$, the result cannot be proved with the same argument, as the exponent on $n$ in the bound above is at best $1/4$, which does not lead to the desired scaling if $q>2$. 
We establish the desired bound for  $d=1$  in Section~\ref{sec:proof_1d}, using a more technical argument.
\end{proof}

We remark that, if $X$ is additionally assumed to be subexponential, then \citet[Corollary 5.2]{lei2020convergence}
establishes exponential tail bounds for $\dw(\Pemp,P)$; under this stronger assumption, the results of Theorem~\ref{thm:main_empirical}
could then be strengthened to give a tail bound for $\dH^2\bigl(\psi^*(\Pemp),\psi^*(P)\bigr)$, in place of the bound on expected
value.

\subsubsection{Lower bounds for the empirical setting}
Our final main result studies the optimality of the power of $n$ appearing in Theorem~\ref{thm:main_empirical}. 
\begin{theorem}\label{thm:lowerbd_empirical}
For any $d \geq 1$ and $q>1$, there exist $\eps^*_d,c_d>0$, depending only on $d$, such
that
\[\sup_{P\in\Pcal_d:\Ep{P}{\norm{X}^q}\leq 1,\, \eps_P\geq \eps^*_d} \EE{\dH^2(\psi^*(\Pemp),\psi^*(P))}\geq c_d \cdot n^{-\min\left\{\frac{2}{d+1},\frac{1}{2}-\frac{1}{2q}\right\}}.\]
\end{theorem}
Ignoring a logarithmic factor in $n$, the first term, namely $n^{-\frac{2}{d+1}}$, is the known minimax rate for {\em any} estimator under
 the well-specified case where $P$ is itself log-concave, for any $d\geq 2$ \citep{kim2016global,kur2019log}.
The second term is a new result and will be proved via a misspecified construction where $P$ is not log-concave: the distribution is given by $X = R\cdot U$, where $U$ is drawn uniformly from the unit sphere $\mathbb{S}_{d-1}$, while the radius $R$ is drawn independently with
\[R = \begin{cases}1/2, & \textnormal{ with probability $1-1/2n$},\\ n^{1/q}, & \textnormal{ with probability $1/2n$.}\end{cases}\]
The intuition is that, with positive probability, the empirical distribution $\Pemp$ (and, therefore, its log-concave projection $\psi^*(\Pemp)$), is supported on the ball of radius $1/2$; on the other hand, we will see in the proof that $\psi^*(P)$ places $\sim n^{-\frac{1}{2}+\frac{1}{2q}}$ mass outside this ball,
leading to a lower bound on the Hellinger distance between these two log-concave projections.

A consequence of this last result
in dimension $d=1$ is that rates of convergence in log-concave density estimation can be much slower in the misspecified setting, 
with a minimax rate of $n^{-1/2}$ at best, as compared to the well-specified setting when $P$ is assumed to have a log-concave 
density, where the corresponding rate is $n^{-4/5}$ \citep{kim2016global}.

\subsubsection{A gap for dimension $d\geq 2$}
Comparing the lower bound established in Theorem~\ref{thm:lowerbd_empirical} with the upper bound given in Theorem~\ref{thm:main_empirical}, we see that for the case $d=1$ the
two bounds match, as they both scale as $n^{-\frac{1}{2}+\frac{1}{2q}}$ (ignoring poly-logarithmic factors).
For $d\geq 2$, however, there is a gap---for sufficiently large $q$ (i.e., a sufficiently strong moment condition), the upper bound scales
as $n^{-\frac{1}{2d}}$ (up to poly-logarithmic factors) while the lower bound has the faster rate $n^{-\frac{2}{d+1}}$. We also remark that the optimal dependence of the minimax rate on $d$ remains unknown as well.  

\section{Relationship with prior work}
\label{Sec:PriorWork}

Log-concave density estimation is a central problem within the field of nonparametric inference under shape constraints.  Entry points to the field include the book by \citet{groeneboom2014nonparametric}, as well as the 2018 special issue of the journal \emph{Statistical Science} \citep{samworth2018special}.  Other important shape-constrained problems that could benefit from the perspective taken in this work include decreasing density estimation \citep{grenander1956theory,rao1969estimation,groeneboom1985estimating,birge1989grenader,jankowski2014convergence}, isotonic regression \citep{brunk1972statistical,zhang2002risk,chatterjee2015risk,durot2018limit,bellec2018sharp,yang2019contraction,han2019isotonic} and convex regression \citep{hildreth1954point,seijo2011nonparametric,cai2015framework,guntuboyina2015global,han2016multivariate,fang2019risk}, among many others.  In these cases, the analysis is likely to be more straightforward, since the canonical least squares/maximum likelihood estimator can be characterised as an $L_2$-projection onto a convex set.  By contrast, the class $\Fcal_d$ is not convex, and the Kullback--Leibler projection $\psi^*$ is considerably more involved.

Early work on log-concave density estimation includes \citet{walther2002detecting}, \citet{pal2007estimating}, \citet{dumbgen2009maximum}, \citet{walther2009inference}, \citet{cule2010maximum}, \citet{cule2010theoretical}, \citet{schuhmacher2011multivariate}, \citet{samworth2012independent} and \citet{chen2013smoothed}.  Sometimes, the class is considered as a special case of the class of $s$-concave densities \citep{koenker2010quasi,seregin2010nonparametric,han2016approximation,doss2016global,han2019global}.  For the case of correct model specification, where~$P$ has density $f_P \in \Fcal_d$ and $\widehat{f}_n := \psi^*(\Pemp)$, it is now known \citep{kim2016global,kur2019log} that
\[
  \sup_{f_P \in \Fcal_d} \EE{\dH^2(\widehat{f}_n,f_P)} \leq K_d \cdot \left\{ \begin{array}{ll} n^{-4/5} & \mbox{when $d=1$} \\
                                                                                             n^{-2/(d+1)}\log n & \mbox{when $d \geq 2$,} \end{array} \right.
\]
where $K_d > 0$ depends only on $d$, and that this risk bound is minimax optimal (up to the logarithmic factor when $d \geq 2$).  See also \citet{carpenter2018near} for an earlier result in the case $d \geq 4$, and \citet{xu2019high} for an alternative approach to high-dimensional log-concave density estimation that seeks to evade the curse of dimensionality in the additional presence of symmetry constraints.  It is further known that when $d \leq 3$, the log-concave maximum likelihood estimator can adapt to certain subclasses of log-concave densities, including log-concave densities whose logarithms are piecewise affine \citep{kim2018adaptation,feng2018adaptation}.  Although these recent works provide a relatively complete picture of the behaviour of the log-concave maximum likelihood estimator when the true distribution has a log-concave density, there is almost no prior work on risk bounds under model misspecification.  The only exception of which we are aware is \citet[][Theorem~1]{kim2018adaptation}, which considers a univariate case where the true distribution has a density that is very close to log-affine on its support.
 
One feature that distinguishes our contributions from earlier work on rates of convergence in log-concave density estimation in the correctly specified setting is that our arguments avoid entirely notions of bracketing entropy, as well as empirical process arguments that control the behaviour of $M$-estimators in terms of the entropy of a relevant function class \citep[e.g.][]{van1996weak,van2000applications}.  It turns out that, for non-convex classes of densities, these ideas are not well suited to the misspecified setting.\footnote{See \citet[][Proposition~4.1]{patilea2001convex} for applications of entropy methods to studying rates of convergence of maximum likelihood estimators for \emph{convex} classes of densities. However, the class of densities $f$ that are log-concave is not a convex class; if we instead consider the class of concave log-densities (i.e., $\log f$, where $f$ is a log-concave density), then 
this class is also not convex, because of the need for the exponentials of these log-densities to integrate to~1.}  Instead, our main tool is a detailed and delicate analysis of the Lipschitz approximations to concave functions introduced in \citet{dumbgen2011approximation}.  In their original usage, these were employed in conjunction with asymptotic results such as Skorokhod's representation theorem to derive the consistency and robustness results described above.  By contrast, our analysis facilitates the direct inequality established in Theorem~\ref{thm:main}.

Another role of this work is to advocate for the benefits of regarding an estimator as a function of the empirical distribution, as opposed to the more conventional view where it is seen as a function on the sample space.  The empirical distribution $\Pemp$ of a sample $X_1,\ldots,X_n$ encodes all of the information in the data when we regard it as a multi-set $\{X_1,\ldots,X_n\}$, i.e.~when we discard information in the ordering of the indices.  It follows that any statistic $\widehat{\theta}_n = \widehat{\theta}_n(X_1,\ldots,X_n)$ that is invariant to permutation of its arguments can be thought of as a functional $\theta(\Pemp)$ of the empirical distribution.  Frequently, the definition of $\theta$ can be extended to a more general class of distributions $\Pcal$, and we may regard $\theta$ as a \emph{projection} from $\Pcal$ onto a model, or parameter space,~$\Theta$.  This perspective, which was pioneered by Richard von Mises in the 1940s \citep{mises1947asymptotic} and described in \citet[][Chapter~6]{serfling1980approximation}, offers many advantages to the statistician.  In particular, once the analytical properties (e.g.~continuity, differentiability) of~$\theta$ are understood, key statistical properties of the estimator (consistency, robustness to misspecification, rates of convergence), can often be deduced as simple corollaries of basic facts about the convergence of empirical distributions.

\section{Proofs of upper bounds}\label{sec:proofs}

In this section we prove Theorem~\ref{thm:main} (for arbitrary dimension $d$), and complete the proof of Theorem~\ref{thm:main_empirical} (for the
remaining case of dimension $d=1$).
In Section~\ref{sec:proof_background} we review some known properties of  log-concave projection, and in Section~\ref{sec:proof_state_Lip_lemma} we establish
a key lemma that will be used in both proofs. In Section~\ref{sec:proof_thm:main} we complete the proof of Theorem~\ref{thm:main},
and in Section~\ref{sec:proof_1d} we complete the proof of Theorem~\ref{thm:main_empirical} for the remaining case $d=1$.

\subsection{Background on  log-concave projection}\label{sec:proof_background}
We begin by reviewing some known properties of  log-concave projection, and computing some new bounds.
\subsubsection{Moment inequalities}
The log-concave projection $\psi^*$ is known to satisfy a useful convex ordering property \citep[Eqn.~(3)]{dumbgen2011approximation}:
for any $P\in\Pcal_d$ and for $f=\psi^*(P)$,\begin{equation}\label{eqn:convexorder}
\Ep{f}{h(X)}\leq\Ep{P}{h(X)}\textnormal{
for any convex function $h:\R^d\rightarrow(-\infty,\infty]$.}
\end{equation}
 In particular, this implies that
\[\Ep{f}{|v^\top (X-\mu_P)|} \leq \Ep{P}{|v^\top (X-\mu_P)|}\textnormal{ for all $v\in\R^d$}.\]

The following lemma establishes that, up to a constant, this inequality is tight for  all vectors $v\in\R^d$.
\begin{lemma}\label{lem:1stmoment}
Fix any $P\in\Pcal_d$, and let $f=\psi^*(P)$. Then 
\[\Ep{f}{|v^\top (X-\mu_P)|} \geq c_d \cdot \Ep{P}{|v^\top (X-\mu_P)|}\textnormal{ for all $v\in\R^d$},\]
where $c_d \in (0,1]$ depends only on $d$.
\end{lemma}

By \citet[Eqn.~(4)]{dumbgen2011approximation},  log-concave projection preserves the mean, i.e.,
\[\mu_P = \Ep{P}{X} =  \Ep{f}{X}.\]
We can also define the covariance matrix $\Sigma = \textnormal{Cov}_f(X)$,
which is finite (since all moments of a log-concave distribution are finite) and strictly positive definite.  
Lemma~\ref{lem:1stmoment} immediately implies bounds on the eigenvalues of $\Sigma$:
\begin{corollary}\label{cor:lambda_min}
Fix any $P\in\Pcal_d$,  let $f=\psi^*(P)$, and let $\Sigma =\textnormal{Cov}_f(X)$ be the covariance matrix of the distribution with density $f$.
Then for all $v\in\R^d$,
\[c_d^2 \bigl\{\Ep{P}{\bigl|v^\top(X-\mu_P)\bigr|}\bigr\}^2 \leq v^\top\Sigma v \leq 16 \bigl\{\Ep{P}{\bigl|v^\top(X-\mu_P)\bigr|}\bigr\}^2,\]
where $c_d\in(0,1]$ is taken from Lemma~\ref{lem:1stmoment}.
In particular, this implies that
\[\lambda_{\min}(\Sigma) \geq (c_d \eps_P)^2,\]
where $\lambda_{\min}(\Sigma)$ denotes the smallest eigenvalue of $\Sigma$.
\end{corollary}
\begin{proof}[Proof of Corollary~\ref{cor:lambda_min}]
First, for the lower bound, by Lemma~\ref{lem:1stmoment} and Cauchy--Schwarz,
  \[c_d^2 \bigl\{\Ep{P}{\bigl|v^\top (X - \mu_P)\bigr|}\bigr\}^2 \leq \bigl\{\Ep{f}{\bigl|v^\top (X - \mu_P)\bigr|}\bigr\}^2 \leq  \Ep{f}{|v^\top (X - \mu_P)|^2}= v^\top \Sigma v .
  \]
Next, for the upper bound,
\[v^\top \Sigma v   = \Ep{f}{|v^\top (X - \mu_P)|^2} \leq 16\bigl\{\Ep{f}{\bigl|v^\top (X - \mu_P)\bigr|}\bigr\}^2 \leq 16\bigl\{\Ep{P}{\bigl|v^\top (X - \mu_P)\bigr|}\bigr\}^2 ,\]
where the first inequality is due to \citet[Theorem 5.22]{lovasz2007geometry}
while the second is by~\eqref{eqn:convexorder} \citep[Eqn.~(3)]{dumbgen2011approximation}.
\end{proof}

\subsubsection{A lower bound on a ball}

Next we show that for any $P$, its log-concave projection $f = \psi^*(P)$ is lower bounded
 on a ball of radius of order $\eps_P$.
\begin{lemma}\label{lem:ball}
Fix any $P\in\Pcal_d$, and let $f=\psi^*(P)$. Then there exist $b_d,r_d \in (0,1]$, depending only on $d$, such that
\[f(x) \geq b_d \cdot \sup_{x' \in \mathbb{R}^d}f(x') \textnormal{ for all $x\in\mathbb{B}_d(\mu_P,r_d\eps_P)$}.\]
\end{lemma}
\begin{proof}[Proof of Lemma~\ref{lem:ball}]
Let $\Sigma = \textnormal{Cov}_f(X)$, and define the isotropic, log-concave density $g(x) = f(\Sigma^{1/2}x + \mu_P) \det^{1/2}(\Sigma)$.  By \citet[Theorem 5.14(a) and~(b)]{lovasz2007geometry},
\[\inf_{x:\norm{x}\leq 1/9}g(x) \geq b_d \sup_{x \in \mathbb{R}^d} g(x),\]
where $b_d\in(0,1]$ depends only on $d$. This immediately implies that
\[f(x) \geq b_d \sup_{x' \in \mathbb{R}^d}f(x')\textnormal{ for all $x\in\R^d$ with $\norm{\Sigma^{-1/2}(x-\mu_P)}\leq 1/9$.}\]
But $\norm{\Sigma^{-1/2}(x-\mu_P)} \leq \lambda_{\min}^{-1/2}(\Sigma)\|x - \mu_P\| \leq \|x - \mu_P\|/(c_d\eps_P)$ by Corollary~\ref{cor:lambda_min}, so the result holds with $r_d = c_d/9$.
\end{proof}

\subsection{Key lemma: the Lipschitz majorization}\label{sec:proof_state_Lip_lemma}
Let
\[\Phi_d := \bigg\{\textnormal{$\phi:\R^d\rightarrow [-\infty,\infty)$ \ : \ \begin{tabular}{c}\textnormal{$\phi$ is a proper concave, upper semi-continuous function,}\\\textnormal{and $\phi(x)\rightarrow-\infty$ as $\norm{x}\rightarrow\infty$}\end{tabular}}\bigg\},\]
and define the function $\phi^*:\Pcal_d\rightarrow\Phi_d$ 
that maps a distribution $P$ to the log-density $\phi = \phi^*(P)$ given by $\phi(x) = \log \big[\psi^*(P)\big](x)$.
\citet[Theorem 2.2]{dumbgen2011approximation} establishes that the log-density $\phi = \phi^*(P)$
maximizes $\ell(\phi,P) := \Ep{P}{\phi(X)} - \int_{\mathbb{R}^d} e^{\phi(x)}\;\mathsf{d}x + 1$ over $\Phi_d$. 
We now show that this maximum can be nearly attained by a Lipschitz function.
In particular, for any $\phi\in \Phi_d$ and any $L>0$, define its $L$-Lipschitz majorization $\phi^L:\mathbb{R}^d \rightarrow \mathbb{R}$ by
\begin{equation}\label{eqn:Lip_maj}\phi^L(x) := \sup_{y\in\R^d} \big\{\phi(y) - L\norm{x-y}\big\}.\end{equation}
It can easily be verified that this function is concave, $L$-Lipschitz, and satisfies $\phi^L(x)\geq \phi(x)$ for all $x \in \mathbb{R}^d$.
Furthermore, it holds that
$\int_{\mathbb{R}^d} e^{\phi^L(x)}\;\mathsf{d}x<\infty$ (this follows from the fact that there exist constants
$a\in\R$, $b>0$ such that $\phi(y)\leq a-b\norm{y}$ for all $y\in\R^d$ \citep{dumbgen2011approximation}), and
moreover $\int_{\mathbb{R}^d} e^{\phi^L(x)}\;\mathsf{d}x>0$.

Next we normalize to produce a log-density. For any $\phi\in\Phi_d$, we define
\begin{equation}\label{eqn:normalize_Lip_maj}\tilde\phi^L(x) := \phi^L(x) - \log\biggl(\int_{\mathbb{R}^d} e^{\phi^L(x)}\;\mathsf{d}x\biggr).\end{equation}

The following result proves that, if $\phi=\phi^*(P)$, then for $L$ sufficiently large, $\tilde\phi^L\in \Phi_d$ is nearly optimal for $P$ (in the sense of maximizing $\ell(\cdot,P)$).
\begin{lemma}\label{lem:Lip_maj}
Fix any $P\in\Pcal_d$, let $\phi = \phi^*(P)$, and let $\phi^L$ and $\tilde\phi^L$ be defined as in~\eqref{eqn:Lip_maj} and~\eqref{eqn:normalize_Lip_maj}. Then for any $L \geq \frac{2d}{r_d\eps_P}$,
\[\ell(\tilde\phi^L,P) \geq \ell(\phi^L,P) \geq \ell(\phi,P) -\frac{4d}{Lb_dr_d\eps_P},\]
where $r_d,b_d \in(0,1]$ are taken from Lemma~\ref{lem:ball}. In particular, this implies that
\[\Ep{P}{\tilde\phi^L(X)} \geq \Ep{P}{\phi(X)}- \frac{4d}{Lb_dr_d\eps_P}.\]
\end{lemma}

\subsubsection{Bounding the Hellinger distance}
Now we apply Lemma~\ref{lem:Lip_maj} to the problem of bounding Hellinger distance.
\begin{corollary}\label{cor:Lip_maj}
Fix any $P,Q\in\Pcal_d$, and define $\eps = \min\{\eps_P,\eps_Q\}>0$. Let $\phi_P=\phi^*(P)$ and $\phi_Q=\phi^*(Q)$, and let $f_P=\psi^*(P)$ and $f_Q=\psi^*(Q)$ be the corresponding
density functions. Let $\phi^L_P$ and $\phi^L_Q$ be the $L$-Lipschitz majorizations of $\phi_P$ and $\phi_Q$, respectively, as defined in~\eqref{eqn:Lip_maj},
for some $L \geq \frac{2d}{r_d\eps}$, where $r_d \in (0,1]$ is taken from Lemma~\ref{lem:ball}. Then
\[\dH^2(f_P,f_Q) \leq \frac{16d}{Lb_dr_d\eps}+ \left(\Ep{P}{\phi^L_P(X)} - \Ep{Q}{\phi^L_P(X)}\right) +  \left(\Ep{Q}{\phi^L_Q(X)}  - \Ep{P}{\phi^L_Q(X)}\right) ,\]
where $b_d  \in (0,1]$ is taken from Lemma~\ref{lem:ball}.
\end{corollary}
\begin{proof}[Proof of Corollary~\ref{cor:Lip_maj}]
Let $\tilde\phi^L_P,\tilde\phi^L_Q$ be defined as in~\eqref{eqn:normalize_Lip_maj}, and let $\tilde{f}^L_P,\tilde{f}^L_Q$ be the corresponding densities,
i.e., $\tilde{f}^L_P(x)=e^{\tilde\phi^L_P(x)}$ and similarly for $\tilde{f}^L_Q$.
We first calculate
\[\dkl(f_P || \tilde{f}^L_P) = \Ep{f_P}{\phi_P(X) - \tilde\phi^L_P(X)} \leq \Ep{P}{\phi_P(X) - \tilde\phi^L_P(X)}\]
and
\[\dkl(f_P || \tilde{f}^L_Q) = \Ep{f_P}{\phi_P(X) - \tilde\phi^L_Q(X)} \leq \Ep{P}{\phi_P(X) - \tilde\phi^L_Q(X)},\]
where the inequalities hold by \citet[Remark 2.3]{dumbgen2011approximation}. The same bounds hold with the roles of $P$ and $Q$ reversed.
Furthermore, by the triangle inequality,
\begin{align*}\dH^2(f_P,f_Q)
&=\frac{1}{2} \dH^2(f_P,f_Q) + \frac{1}{2}\dH^2(f_P,f_Q) \\
&\leq \frac{1}{2}\bigl\{\dH(f_P,\tilde{f}^L_P)+\dH(f_Q,\tilde{f}^L_P)\bigr\}^2 + \frac{1}{2}\bigl\{\dH(f_P,\tilde{f}^L_Q)+\dH(f_Q,\tilde{f}^L_Q)\bigr\}^2\\
 &\leq \dH^2(f_P,\tilde{f}^L_P) + \dH^2(f_Q,\tilde{f}^L_P) + \dH^2(f_P,\tilde{f}^L_Q) + \dH^2(f_Q,\tilde{f}^L_Q) \\
  &\leq \dkl(f_P||\tilde{f}^L_P) + \dkl(f_Q||\tilde{f}^L_P) +\dkl(f_P||\tilde{f}^L_Q) + \dkl(f_Q||\tilde{f}^L_Q),\end{align*}
where the last step holds  by the standard inequality relating KL divergence with Hellinger distance (i.e., $\dH^2 \leq \dkl$).
Combining all these calculations, and then rearranging terms, we see that\footnote{All expectations in this display are finite, because, e.g., $\sup_{x \in \mathbb{R}^d} \phi_P(x) = \sup_{x \in \mathbb{R}^d} \phi_P^L(x) < \infty$; moreover, $\Ep{P}{\phi_P^L(X)} \geq \Ep{P}{\phi_P(X)} > -\infty$ because $P \in \mathcal{P}_d$, and $\Ep{P}{\phi_Q^L(X)} > -\infty$ because $\phi_Q^L$ is Lipschitz and $P$ has a finite first moment.} 
\begin{align*}\dH^2(f_P,f_Q) &\leq
                                 \Ep{P}{\phi_P(X) - \tilde\phi^L_P(X)} + \Ep{Q}{\phi_Q(X) - \tilde\phi^L_P(X)} \\
  &\hspace{1in}+ \Ep{P}{\phi_P(X) - \tilde\phi^L_Q(X)} + \Ep{Q}{\phi_Q(X) - \tilde\phi^L_Q(X)}\\
&= 2\left(\Ep{P}{\phi_P(X) - \tilde\phi^L_P(X)} +\Ep{Q}{\phi_Q(X) - \tilde\phi^L_Q(X)} \right) \\&\hspace{1in}{}+ \left(\Ep{P}{\tilde\phi^L_P(X)} - \Ep{Q}{\tilde\phi^L_P(X)}\right) +  \left(\Ep{Q}{\tilde\phi^L_Q(X)}  - \Ep{P}{\tilde\phi^L_Q(X)}\right) \\
&= 2\left(\Ep{P}{\phi_P(X) - \tilde\phi^L_P(X)} +\Ep{Q}{\phi_Q(X) - \tilde\phi^L_Q(X)} \right) \\&\hspace{1in}{}+ \left(\Ep{P}{\phi^L_P(X)} - \Ep{Q}{\phi^L_P(X)}\right) +  \left(\Ep{Q}{\phi^L_Q(X)}  - \Ep{P}{\phi^L_Q(X)}\right) ,\end{align*}
where the last step holds since $\tilde\phi^L_P,\tilde\phi^L_Q$ are simply shifts of the functions $\phi^L_P,\phi^L_Q$, respectively.
Finally, applying Lemma~\ref{lem:Lip_maj} concludes the proof.
\end{proof}

\subsection{Completing the proof of Theorem~\ref{thm:main}}\label{sec:proof_thm:main}
We will now apply Corollary~\ref{cor:Lip_maj} to prove Theorem~\ref{thm:main}, bounding $\dH^2(f_P,f_Q)$ in terms of the Wasserstein distance. 
Define
\[L = \sqrt{\frac{8d}{r_db_d\min\{\eps_P,\eps_Q\} \dw(P,Q)} },\]
where $r_d, b_d \in (0,1]$ are taken from Lemma~\ref{lem:ball}.
Take a coupling $(X,Y)$ of $d$-dimensional random vectors with marginal distributions $X\sim P$ and $Y\sim Q$, such that $\EE{\norm{X-Y}} = \dw(P,Q)$, which is guaranteed to exist by \citet[][Theorem~4.1]{villani2008optimal}.
Then, since $\phi^L_P$ is $L$-Lipschitz, we have
\[\EE{\phi^L_P(X)} - \EE{\phi^L_P(Y)} \leq \EE{L\norm{X-Y}} = L\dw(P,Q),\]
and similarly
\[\EE{\phi^L_Q(Y)} - \EE{\phi^L_Q(X)} \leq L\dw(P,Q).\]
If $L\geq  \frac{2d}{r_d\min\{\eps_P,\eps_Q\}}$, then  applying Corollary~\ref{cor:Lip_maj}, we have 
\[\dH^2\bigl(\psi^*(P),\psi^*(Q)\bigr) \leq \frac{16d}{Lb_dr_d\min\{\eps_P,\eps_Q\}}+ 2L\dw(P,Q) =  \sqrt{\frac{128d\dw(P,Q)}{r_db_d\min\{\eps_P,\eps_Q\}}}.\]
If instead $L< \frac{2d}{r_d\min\{\eps_P,\eps_Q\}}$, then $ \frac{db_d\dw(P,Q)}{2r_d\min\{\eps_P,\eps_Q\}} > 1$.
Since Hellinger distance is always bounded by $\sqrt{2}$, we then have
\[\dH^2(\psi^*(P),\psi^*(Q)) \leq 2 \leq \sqrt{\frac{2db_d\dw(P,Q)}{r_d\min\{\eps_P,\eps_Q\}}} \leq \sqrt{\frac{2d\dw(P,Q)}{r_db_d\min\{\eps_P,\eps_Q\}}},\]
where the last step holds trivially since $b_d\leq 1$.
Thus, in either case, we have 
\[\dH^2\bigl(\psi^*(P),\psi^*(Q)\bigr) \leq \sqrt{\frac{128d}{r_db_d}} \cdot \sqrt{\frac{\dw(P,Q)}{\min\{\eps_P,\eps_Q\}}}.\]
We now split into cases. If $\dw(P,Q)\leq \max\{\eps_P,\eps_Q\}/4$, then
\[\frac{\dw(P,Q)}{\min\{\eps_P,\eps_Q\}} =\frac{\dw(P,Q)}{\max\{\eps_P,\eps_Q\} - |\eps_P-\eps_Q|} \leq\frac{\dw(P,Q)}{\max\{\eps_P,\eps_Q\} - 2\dw(P,Q)} \leq \frac{2\dw(P,Q)}{\max\{\eps_P,\eps_Q\} },\]
where the second step applies Proposition~\ref{prop:eps_P}. If instead $\dw(P,Q)> \max\{\eps_P,\eps_Q\}/4$
then we will instead use the trivial bound
\[
  \dH^2\bigl(\psi^*(P),\psi^*(Q)\bigr) \leq 2 \leq 4 \sqrt{\frac{\dw(P,Q)}{\max\{\eps_P,\eps_Q\}}} \leq 4\sqrt{\frac{d}{r_db_d}} \cdot \sqrt{\frac{\dw(P,Q)}{\max\{\eps_P,\eps_Q\}}}
\]
where the last step is trivial since $d\geq 1$ and $r_d,b_d\in(0,1]$.  Thus, in both cases, we have 
\[\dH^2\bigl(\psi^*(P),\psi^*(Q)\bigr) \leq 16\sqrt{\frac{d}{r_db_d}} \cdot \sqrt{\frac{\dw(P,Q)}{\max\{\eps_P,\eps_Q\}}}.\]
This proves the theorem, when we choose $C_d =4 \bigl(\frac{d}{r_db_d}\bigr)^{1/4}$.

\subsection{Completing the proof of Theorem~\ref{thm:main_empirical}: the case $d=1$}\label{sec:proof_1d}
Before proving the theorem, we first state several supporting lemmas. First we state a deterministic result:
\begin{lemma}\label{lem:1d}
  Let $P,Q\in\Pcal_1$ satisfy  $\max\{\Ep{P}{|X|^q}^{1/q},\Ep{Q}{|X|^q}^{1/q}\}\leq M_q$ for some $q>1$.  Define
  \begin{align*}
    \Delta_{\textnormal{CDF}}&(P,Q) \\
    &:= \max\left\{\sup_{t\in\R}\left|\sqrt{\Pp{P}{X>t}}-\sqrt{\Pp{Q}{X>t}}\right| , \sup_{t\in\R}\left|\sqrt{\Pp{P}{X<t}}-\sqrt{\Pp{Q}{X<t}}\right|\right\}.\end{align*}
Then
\[\dH^2\bigl(\psi^*(P),\psi^*(Q)\bigr)\leq C_*\sqrt{\frac{M_q}{\max\{\eps_P,\eps_Q\}}} \cdot \Big\{\Delta_{\textnormal{CDF}}(P,Q)\cdot\log\bigl(e/\Delta_{\textnormal{CDF}}(P,Q)\bigr)\Big\}^{1-1/q},\]
for a universal constant $C_* > 0$.
\end{lemma}
\noindent Next, in order to prove Theorem~\ref{thm:main_empirical}, we will want to apply this result with $Q=\Pemp$, i.e., we want to 
bound $\Delta_{\textnormal{CDF}}(\Pemp,P)$. 
Let $F$ denote the distribution function of $P$, and, for $t \in (0,1)$, let $F^{-1}(t) := \inf\{x:F(x) \geq t\}$.  Then, with $U\sim\textnormal{Unif}[0,1]$, we know that $F^{-1}(U)\sim P$. 
We may therefore assume that $X_1,\ldots,X_n$ are generated as $X_i=F^{-1}(U_i)$, where $U_1,\dots,U_n\iidsim\textnormal{Unif}[0,1]$. Since $F^{-1}$ is monotonic, we have
\begin{equation}\label{eqn:reduce_cdf_to_uniform}\Delta_{\textnormal{CDF}}(\Pemp,P)\leq \Delta_{\textnormal{CDF}}\bigl(\widehat{U}_n,\textnormal{Unif}[0,1]\bigr),\end{equation}
where $\widehat{U}_n$ is the empirical distribution of $U_1,\dots,U_n$.  Therefore, it suffices to consider the case that $P$ is the uniform distribution.
We now apply results from \citet{shorack2009empirical} to prove a tail bound on $\Delta_{\textnormal{CDF}}(\widehat{U}_n,\textnormal{Unif}[0,1])$.
\begin{lemma}\label{lem:cdf}
Fix any $n\geq 2$, and let
$\widehat{U}_n$ be the empirical distribution of $U_1,\dots,U_n\iidsim\textnormal{Unif}[0,1]$.
Then, for any $c>0$, 
\[\PP{\Delta_{\textnormal{CDF}}(\widehat{U}_n,\textnormal{Unif}[0,1]) \leq c'\sqrt{\frac{\log n}{n}}} \geq 1 - n^{-c},\]
where $c'> 0$ depends only on $c$.
\end{lemma}

With these lemmas in place, we are now in a position to prove Theorem~\ref{thm:main_empirical}.
Let $M_{q,n} = \left(\frac{1}{n}\sum_{i=1}^n |X_i|^q\right)^{1/q}$ and $\Delta = \Delta_{\textnormal{CDF}}(\Pemp,P)$.
If $\Pemp\in\Pcal_1$ (that is, $\Pemp$ does not place all its mass on a single point), then we have
\begin{equation}\label{eqn:dH_emp_for_d_equals_1}\dH^2\bigl(\psi^*(\Pemp),\psi^*(P)\bigr)\leq\min\biggl\{2,C_*\sqrt{\frac{\max\{M_q,M_{q,n}\}}{\max\{\eps_P,\eps_{\Pemp}\}}} \cdot \big(\Delta\log(e/\Delta)\big)^{1-1/q}\biggr\}\end{equation}
by applying Lemma~\ref{lem:1d}  with $Q=\Pemp$.
On the other hand, if $\Pemp$ does place all its mass on one point, then recall that $\psi^*(\Pemp)$ is not defined but
we take $\dH^2\bigl(\psi^*(\Pemp),\psi^*(P)\bigr) = 2$ by convention. For this case, we can trivially calculate
\[\Delta \geq  \min\bigl\{\sqrt{\Pp{P}{X>\mu_P}},\sqrt{\Pp{P}{X<\mu_P}}\bigr\}.\]
We will now need an additional lemma:
\begin{lemma}\label{lem:prob_above_below_mu}
Fix any $P\in\Pcal_1$ and any $q > 1$. Suppose $M_q = \Ep{P}{|X|^q}^{1/q}<\infty$. Then
\[\min\left\{\Pp{P}{X>\mu_P},\Pp{P}{X<\mu_P}\right\}\geq \left(\frac{\eps_P}{4M_q}\right)^{\frac{q}{q-1}}.\]
\end{lemma}
\noindent This implies
\[\Delta \geq \left(\frac{\eps_P}{4M_q}\right)^{\frac{q}{2(q-1)}} \]
for the case where $\Pemp\not\in\Pcal_1$ (i.e., $\Pemp$ is supported on a single point). Since also $\Delta\leq1$ by definition, this means that
\[ \sqrt{\frac{\max\{M_q,M_{q,n}\}}{\eps_P}} \cdot \big(\Delta\log(e/\Delta)\big)^{1-1/q} \geq \frac{1}{2} = \frac{\dH^2\bigl(\psi^*(\Pemp),\psi^*(P)\bigr)}{4}.\]
Combining this with~\eqref{eqn:dH_emp_for_d_equals_1} for the case $\Pemp\in\Pcal_1$, we see that 
\[\dH^2\bigl(\psi^*(\Pemp),\psi^*(P)\bigr)\leq\min\biggl\{2,\max\{C_*,4\}\sqrt{\frac{\max\{M_q,M_{q,n}\}}{\eps_P}} \cdot \big(\Delta\log(e/\Delta)\big)^{1-1/q}\biggr\}\]
holds for both cases.

Next, we will combine this calculation with Lemma~\ref{lem:cdf}, applied with $c= 1/2$.  Let $c'$ be the constant from Lemma~\ref{lem:cdf}. First, if $c'\sqrt{\frac{\log n}{n}}>1$, 
then 
\begin{multline*}\EE{\dH^2\bigl(\psi^*(\Pemp),\psi^*(P)\bigr)}\leq 2 \leq 2\biggl(c'\sqrt{\frac{\log n}{n}}\biggr)^{1-1/q} \leq 
\frac{2c'^{1-1/q}}{(\log 2)^{1-1/q}} \frac{\log^{\frac{3}{2}(1-1/q)}n}{n^{\frac{1}{2}-\frac{1}{2q}}}\\
\leq \frac{2c'^{1-1/q}}{(\log 2)^{1-1/q}} \cdot \sqrt{\frac{2M_q}{\eps}}\cdot \frac{\log^{\frac{3}{2}(1-1/q)}n}{n^{\frac{1}{2}-\frac{1}{2q}}},\end{multline*}
where the last step holds since
\begin{equation}\label{eqn:eps_vs_Mq}\eps_P = \Ep{P}{|X-\mu_P|} \leq \Ep{P}{|X|} + |\mu_P| \leq 2\Ep{P}{|X|} \leq 2\bigl\{\Ep{P}{|X|^q}\bigr\}^{1/q} \leq 2M_q.
\end{equation}
If instead $c'\sqrt{\frac{\log n}{n}} \leq 1$, then we have
\begin{align*}
&\EE{\dH^2\bigl(\psi^*(\Pemp),\psi^*(P)\bigr)}\\
&\leq \EE{ \min\biggl\{2,\max\{C_*,4\}\sqrt{\frac{\max\{M_q,M_{q,n}\}}{\max\{\eps_P,\eps_{\Pemp}\}}} \cdot \big(\Delta\log(e/\Delta)\big)^{1-1/q}\biggr\}}\\
&\leq 2\PP{\Delta>c'\sqrt{\frac{\log n}{n}}} + \EE{\max\{C_*,4\}\sqrt{\frac{M_q+M_{q,n}}{\eps_P}}\cdot \Biggl\{c'\sqrt{\frac{\log n}{n}}\log\Biggl(\frac{e}{c'\sqrt{\frac{\log n}{n}}}\Biggr)\Biggr\}^{1-1/q}}\\
&\leq 2n^{-1/2} + \max\{C_*,4\}\sqrt{\frac{M_q+\EE{M_{q,n}}}{\eps_P}}\cdot \Biggl\{c'\sqrt{\frac{\log n}{n}}\log\Biggl(\frac{e}{c'\sqrt{\frac{\log n}{n}}}\Biggr)\Biggr\}^{1-1/q}\\
&\leq 2n^{-1/2} + \max\{C_*,4\}\sqrt{\frac{2M_q}{\eps_P}}\cdot \Biggl\{c'\sqrt{\frac{\log n}{n}}\log\Biggl(\frac{e}{c'\sqrt{\frac{\log n}{n}}}\Biggr)\Biggr\}^{1-1/q}\\
&\leq \sqrt{\frac{2M_q}{\eps_P}}\cdot \left[2n^{-1/2} + \max\{C_*,4\}\Biggl\{c'\sqrt{\frac{\log n}{n}}\log\Biggl(\frac{e}{c'\sqrt{\frac{\log n}{n}}}\Biggr)\Biggr\}^{1-1/q}\right],
\end{align*}
where the third-to-last step applies Jensen's inequality,
the second-to-last step holds  because $\EE{M_{q,n}}\leq M_q$, and the last step holds by~\eqref{eqn:eps_vs_Mq}.
After simplifying, we obtain
 \[\EE{\dH^2\bigl(\psi^*(\Pemp),\psi^*(P)\bigr)}\leq C_{1,q}\sqrt{\frac{M_q}{\eps_P}}\cdot \frac{\log^{\frac{3}{2}(1-1/q)} n}{n^{\frac{1}{2}-\frac{1}{2q}}}\]
for all $n \geq 2$ when $C_{1,q}$ is chosen appropriately. This completes the proof of Theorem~\ref{thm:main_empirical} for the case $d=1$.

\appendix
\section{Additional proofs}\label{app:proofs}

\subsection{Proof of Proposition~\ref{prop:eps_P}}
First fix any distribution $P$ on $\R^d$ with $\Ep{P}{\norm{X}}<\infty$.  Observe that $u\mapsto\Ep{P}{|u^\top(X-\mu_P)|}$ is a continuous function on $\mathbb{S}_{d-1}$, since for any $u,v\in\mathbb{S}_{d-1}$, we have 
\begin{multline*}\left|\Ep{P}{|u^\top(X-\mu_P)|} - \Ep{P}{|v^\top(X-\mu_P)|}\right| \leq \Ep{P}{\big|(u-v)^\top (X-\mu_P)\big|}\\ \leq \norm{u-v} \cdot \Ep{P}{\norm{X-\mu_P}}
 \leq \norm{u-v} \cdot 2\Ep{P}{\norm{X}},\end{multline*}
and $\Ep{P}{\norm{X}}<\infty$ by assumption.
Therefore, $u \mapsto \Ep{P}{ \big|u^\top(X-\mu_P)\big|}$ must attain its infimum, that is,
\[\eps_P =  \inf_{u\in\mathbb{S}_{d-1}} \Ep{P}{ \big|u^\top(X-\mu_P)\big|} =  \Ep{P}{ \big|u_0^\top(X-\mu_P)\big|} \]
for some $u_0\in\mathbb{S}_{d-1}$. 

Next suppose $P\in\Pcal_d$. We will show that $\eps_P>0$. As above, we have $\eps_P=  \Ep{P}{ \big|u_0^\top(X-\mu_P)\big|}$
for some $u_0\in\mathbb{S}_{d-1}$. If $\eps_P=0$, then this implies that $u_0^\top(X-\mu_P)=0$ with probability $1$, meaning that $P$ places all its mass on a single hyperplane $H = \{x \in\R^d: u_0^\top x = u_0^\top \mu_P\}$.
This contradicts the assumption $P\in\Pcal_d$, thus proving the first claim.

Finally, consider distributions $P,Q$ on $\R^d$ with $\Ep{P}{\norm{X}},\Ep{Q}{\norm{X}}<\infty$. By \citet[][Theorem~4.1]{villani2008optimal}, we can find a pair of $d$-dimensional random vectors $X$ and $Y$ such that marginally $X\sim P$, $Y\sim Q$ and $\EE{\norm{X-Y}}=\dw(P,Q)$. Let $u_0$ be defined as above, so that $\eps_P=  \EE{ \big|u_0^\top(X-\mu_P)\big|}$.  Then
\begin{align*}
\eps_Q - \eps_P
&= \inf_{u\in\mathbb{S}_{d-1}} \EE{ \big|u^\top(Y-\mu_Q)\big|} - \EE{ \big|u_0^\top(X-\mu_P)\big|}\\
&\leq  \EE{ \big|u_0^\top(Y-\mu_Q)\big|} - \EE{ \big|u_0^\top(X-\mu_P)\big|}\\
&\leq \EE{\big|u_0^\top(X-Y)\big|} + \big|u_0^\top (\mu_P - \mu_Q)\big|\\
&\leq \EE{\norm{X-Y}} + \norm{\mu_P-\mu_Q}\\
&\leq 2\EE{\norm{X-Y}}\\
&=2\dw(P,Q).
\end{align*}
An identical argument proves the reverse bound, and we deduce that $|\eps_P-\eps_Q|\leq 2\dw(P,Q)$, as desired.

\subsection{Proof of Lemma~\ref{lem:1stmoment}}
Let $\Sigma = \textnormal{Cov}_f(X)$ and define an isotropic log-concave density $g$ on $\R^d$ by $g(x) = f(\Sigma^{1/2}x +\mu_P) \det^{1/2}(\Sigma)$. Note that, if $X\sim f$, then $\Sigma^{-1/2}(X-\mu_P)\sim g$.  Hence
\begin{align*}\Ep{f}{|v^\top (X-\mu_P)|} &= \Ep{f}{|(\Sigma^{1/2}v)^\top (\Sigma^{-1/2}(X-\mu_P))|} = \Ep{g}{ |(\Sigma^{1/2}v)^\top X|}\\
&\geq \frac{1}{4} \big(\Ep{g}{((\Sigma^{1/2}v)^\top X)^2}\big)^{1/2} = \frac{1}{4} \norm{\Sigma^{1/2}v},\end{align*}
where the inequality applies \citet[Theorem 5.22]{lovasz2007geometry}, and the last step holds because $g$ is isotropic.

Next, define a distribution $Q$ obtained by drawing $X\sim P$ and then taking the affine transformation $\Sigma^{-1/2}(X-\mu_P)$.
By definition of $Q$, we have
\[\Ep{P}{|v^\top (X-\mu_P)|} = \Ep{P}{|(\Sigma^{1/2}v)^\top (\Sigma^{-1/2}(X-\mu_P))|}  = \Ep{Q}{ |(\Sigma^{1/2}v)^\top X|}\\
\leq \norm{\Sigma^{1/2}v}\cdot\Ep{Q}{\norm{X}}.\]
Since  log-concave projection commutes with affine transformations, we have
\[ \psi^*(Q) = g,\]
which is an isotropic log-concave density.  Lemma~\ref{lem:isotropic} below establishes that $\Ep{Q}{\norm{X}} \leq a_d$, where $a_d > 0$ depends only on $d$.
Therefore, we have proved that, for any $v\in\R^d$,
\[\Ep{P}{|v^\top (X-\mu_P)|}  \leq \norm{\Sigma^{1/2}v}\cdot a_d\]
while
\[\Ep{f}{|v^\top (X-\mu_P)|}  \geq \frac{1}{4} \norm{\Sigma^{1/2}v}.\]
Setting $c_d = \frac{1}{4a_d}$ establishes the desired result.

\subsubsection{Supporting lemma for Lemma~\ref{lem:1stmoment}}
\begin{lemma}\label{lem:isotropic}
There exists $a_d > 0$, depending only on $d$,
such that, for any isotropic log-concave density $f$ on $\R^d$ and any $P\in\Pcal_d$ with $\psi^*(P)=f$,
\[\Ep{P}{\norm{X}} \leq a_d.\]
\end{lemma}
\begin{proof}[Proof of Lemma~\ref{lem:isotropic}]
By \citet[Lemma 13]{fresen2013multivariate}, since $f$ is an isotropic log-concave density, it holds that
\[f(x) \leq e^{\beta_d - \alpha_d \norm{x}}\textnormal{ for all $x\in\R^d$},\]
where $\alpha_d>0$ and $\beta_d \in \mathbb{R}$ depend only on $d$.
We can therefore calculate
\[\Ep{P}{\log f(X)} \leq \Ep{P}{\beta_d - \alpha_d \norm{X}} = \beta_d - \alpha_d\Ep{P}{\norm{X}}.\]
On the other hand, consider the log-concave density 
\[g(x) =    \biggl(\frac{d^d}{\Ep{P}{\norm{X}}^d(d-1)!S_{d-1}}\biggr) \cdot \exp\biggl\{-\frac{d\norm{x}}{\Ep{P}{\norm{X}}}\biggr\},\]
where $S_{d-1}$ denotes the surface area of the unit sphere $\mathbb{S}_{d-1}$ in $\R^d$ (with $S_0=2$).
We have
\[\Ep{P}{\log g(X)} = \log\biggl(\frac{d^d}{\Ep{P}{\norm{X}}^d(d-1)!S_{d-1}}\biggr) - d.\]
But, since $f = \psi^*(P)$, it must hold that
\[\Ep{P}{\log f(X)} \geq \Ep{P}{\log g(X)},\]
and so
\[ \beta_d - \alpha_d\Ep{P}{\norm{X}} \geq  \log\biggl(\frac{(d/e)^d}{(d-1)!S_{d-1}}\biggr) - d\log \Ep{P}{\norm{X}}.\]
The result follows.
\end{proof}

\subsection{Proof of Lemma~\ref{lem:Lip_maj}}\label{sec:proof_Lip_lemma}
We will prove below that, when $L\geq \frac{2d}{r_d\eps_P}$, the function $\phi^L(x) =\sup_{y \in \mathbb{R}^d}\{\phi(x)-L\norm{x-y}\}$ satisfies
\begin{equation}\label{eqn:Lip_ball_inflate}\int_{\R^d}e^{\phi^L(x)}\;\mathsf{d}x  \leq 1 + \frac{4d}{Lb_dr_d\eps_P}.\end{equation}
Assuming this holds, we then have
\begin{multline*}\ell(\phi^L,P) = \Ep{P}{\phi^L(X)} - \int_{\R^d}e^{\phi^L(x)}\;\mathsf{d}x + 1 \geq  \Ep{P}{\phi^L(X)} - \frac{4d}{Lb_dr_d\eps_P}\\
 \geq  \Ep{P}{\phi(X)} - \frac{4d}{Lb_dr_d\eps_P} = \ell(\phi,P) - \frac{4d}{Lb_dr_d\eps_P},\end{multline*}
 where the last inequality holds since $\phi^L\geq \phi$ pointwise.
Finally, normalizing to $\tilde\phi^L$ can only improve the objective function, since
\[\ell(\tilde\phi^L,P) =  \Ep{P}{\tilde\phi^L(X)}  =  \Ep{P}{\phi^L(X)} - \log\left(\int_{\R^d}e^{\phi^L(x)}\;\mathsf{d}x\right) \geq  \ell(\phi^L,P),\]
because $\log t \leq t-1$ for all $t>0$.

From this point on, we only need to prove~\eqref{eqn:Lip_ball_inflate} in order to complete the proof of the lemma.
For any $x\in\R^d$, we will write $y_x$ to denote a point attaining the supremum, i.e., $\phi^L(x) = \phi(y_x) - L\norm{x-y_x}$
(Lemma~\ref{lem:sup} below verifies the existence and measurability of such a map $x\mapsto y_x$).

We now derive the desired bound~\eqref{eqn:Lip_ball_inflate}. We have
\begin{align*}
\int_{\R^d}e^{\phi^L(x)}\;\mathsf{d}x
&=\int_{\R^d}e^{\phi(y_x)}\cdot e^{ - L\norm{x-y_x}}\;\mathsf{d}x\\
&=\int_{\R^d}\biggl(\int_{-\infty}^{\phi(y_x)}e^t\;\mathsf{d}t \biggr)\cdot \biggl(\int_{L\norm{x-y_x}}^{\infty} e^{-s}\;\mathsf{d}s\biggr)\;\mathsf{d}x\\
&=\int_{-\infty}^{M_\phi}\int_{0}^{\infty}e^{t-s} \biggl(\int_{\R^d} \One{\phi(y_x)\geq t, \norm{x-y_x}\leq s/L}\;\mathsf{d}x\biggr)\;\mathsf{d}s\;\mathsf{d}t,
\end{align*}
where the last step follows by Fubini's theorem, and where
 $M_\phi=\sup_{x\in\R^d}\phi(x)$
(note that we must have $M_\phi<\infty$ by definition of $\Phi_d$).  We now examine this indicator function.
For $t \in \mathbb{R}$ define the super-level set $D_t= \{x : \phi(x)\geq  t\}$. Note that $D_t$ is convex for any $t$ by concavity of $\phi$,
and furthermore is bounded since $\phi$ is a log-density. Moreover, we can observe
that $D_t$ has non-empty interior for any $t<M_\phi$, since $\phi$ is concave and is a log-density.

Now, for any compact, convex set $C \subseteq \R^d$ and any $\delta>0$, define  the $\delta$-neighborhood of $C$ by
\[\textnormal{Nbd}(C,\delta) := \{x\in\R^d: \textnormal{dist}(x,C)\leq \delta\},\]
where $\textnormal{dist}(x,C) := \min_{y\in C}\norm{x-y}$.  (If $C$ is the empty set, then this neighborhood is also defined to be the empty set.)  
If $x \in \R^d$ is such that $\phi(y_x)\geq t$, then $y_x\in D_t$, and if, furthermore, $\norm{x-y_x}\leq  s/L$, then
\[x\in\textnormal{Nbd}(D_t,s/L).\]
Hence,
\[\int_{\R^d}e^{\phi^L(x)}\;\mathsf{d}x \leq \int_{-\infty}^{M_\phi}\int_{0}^{\infty}e^{t-s}\cdot \Leb_d\Big(\textnormal{Nbd}(D_t, s/L)\Big)\;\mathsf{d}s\;\mathsf{d}t.\]
 On the other hand, we have
\begin{align}
\int_{-\infty}^{M_\phi}\int_{0}^{\infty}e^{t-s}\cdot \Leb_d(D_t)\;\mathsf{d}s\;\mathsf{d}t
\notag&=\int_{-\infty}^{M_\phi}e^t\cdot \Leb_d(D_t)\;\mathsf{d}t \nonumber 
=\int_{-\infty}^{M_\phi}e^t \biggl(\int_{\R^d} \One{\phi(x)\geq t}\;\mathsf{d}x\biggr)\;\mathsf{d}t \\
\label{eqn:t_integral_no_inflate}&=\int_{\R^d}\int_{-\infty}^{\phi(x)} e^t\;\mathsf{d}t\;\mathsf{d}x
=\int_{\R^d}e^{ \phi(x)}\;\mathsf{d}x = 1,
\end{align}
by again applying Fubini's theorem. Therefore, to prove~\eqref{eqn:Lip_ball_inflate}, we only need to show that
\begin{equation}\label{eqn:Lip_ball_inflate_diff}
\int_{-\infty}^{M_\phi}\int_{0}^{\infty}e^{t-s}\cdot \Leb_d\Big(\textnormal{Nbd}(D_t, s/L)\backslash D_t\Big)\;\mathsf{d}s\;\mathsf{d}t
\leq  \frac{4d}{Lb_dr_d\eps_P}.\end{equation}

 Next we will use a basic result about neighborhoods of convex sets---Lemma~\ref{lem:convex_inflate}
verifies that 
\[\delta\mapsto \frac{\Leb_d\big(\textnormal{Nbd}(C,\delta)\backslash C\big)}{\delta}\]
is a non-decreasing function for any compact, convex set $C \subseteq \R^d$ with non-empty interior. 
Therefore, for any $t<M_\phi$, it holds that
\[\Leb_d\Big(\textnormal{Nbd}(D_t, s/L)\backslash D_t\Big)\leq  \frac{2d}{Lr_d\eps_P} \cdot \Leb_d\Big(\textnormal{Nbd}\Big(D_t, \frac{sr_d\eps_P}{2d}\Big)\backslash D_t\Big)\]
since we have assumed $L\geq \frac{2d}{r_d\eps_P}$. 
We also have $D_t \subseteq D_{t + \log b_d}$, where $b_d \in(0,1]$ is the constant appearing in  Lemma~\ref{lem:ball}, and so
\[\Leb_d\Big(\textnormal{Nbd}\Big(D_t, \frac{sr_d\eps_P}{2d}\Big)\backslash D_t\Big) \leq \Leb_d\Big(\textnormal{Nbd}\Big(D_t, \frac{sr_d\eps_P}{2d}\Big)\Big) \\
\leq \Leb_d\Big(\textnormal{Nbd}\Bigl(D_{t + \log b_d}, \frac{sr_d\eps_P}{2d}\Big)\Big).\]
Recall from Lemma~\ref{lem:ball} that $D_{M_\phi + \log b_d}$ contains $ \mathbb{B}_d(\mu_P,r_d\eps_P)$. Therefore,
for any $t<M_\phi$, $D_{t+\log b_d}\supseteq D_{M_\phi + \log b_d}$ also contains this ball, and so 
\begin{align*}
  \textnormal{Nbd}\Big(D_{t + \log b_d}, \frac{sr_d\eps_P}{2d}\Big) &= D_{t + \log b_d} + \frac{s}{2d} \cdot  \mathbb{B}_d(\mu_P,r_d\eps_P)\\
                                                                          &\subseteq  D_{t + \log b_d} + \frac{s}{2d} \cdot D_{t+\log b_d} \\
  &= \biggl(1+ \frac{s}{2d} \biggr)  \cdot D_{t+\log b_d} ,\end{align*}
where for two sets $A,B\subseteq\R^d$, we write $A+B:=\{x+y:x\in A, y\in B\}$ to denote their Minkowski sum. Therefore,
\[\Leb_d\Big(\textnormal{Nbd}\Big(D_{t + \log b_d}, \frac{sr_d\eps_P}{2d}\Big) \Big) \leq \Leb_d(D_{t+\log b_d})\cdot  \left(1+ \frac{s}{2d} \right)^d \leq 
\Leb_d(D_{t+\log b_d}) \cdot e^{s/2}\]
for any $t<M_\phi$. Combining this with our work above, we obtain
\begin{equation}\label{eqn:volume_inflate}
\Leb_d\Big(\textnormal{Nbd}(D_t, s/L)\backslash D_t\Big)\leq
  \frac{2d}{Lr_d\eps_P} \cdot
\Leb_d(D_{t+\log b_d}) \cdot e^{s/2}\end{equation}
for any $t< M_\phi$.
Therefore,
\begin{align*}
\int_{-\infty}^{M_\phi}\int_{0}^{\infty}&e^{t-s}\cdot \Leb_d\Big(\textnormal{Nbd}(D_t, s/L)\backslash D_t\Big)\;\mathsf{d}s\;\mathsf{d}t\\
&\leq\int_{-\infty}^{M_\phi}\int_{0}^{\infty}e^{t-s}\cdot \frac{2d}{Lr_d\eps_P} \cdot
\Leb_d(D_{t+\log b_d}) \cdot e^{s/2}\;\mathsf{d}s\;\mathsf{d}t\\
&=\frac{2d}{Lr_d\eps_P} \cdot\left(\int_{-\infty}^{M_\phi}e^t\cdot 
\Leb_d(D_{t+\log b_d})\;\mathsf{d}t\right)\cdot \left(\int_{0}^{\infty} e^{-s}\cdot e^{s/2}\;\mathsf{d}s\right)\\
&=\frac{4d}{Lr_d\eps_P} \cdot \int_{-\infty}^{M_\phi}e^t\cdot 
\Leb_d(D_{t+\log b_d})\;\mathsf{d}t\\
&=\frac{4d}{Lb_dr_d\eps_P} \cdot \int_{-\infty}^{M_\phi}e^{t+\log b_d}\cdot 
\Leb_d(D_{t+\log b_d})\;\mathsf{d}t\\
&=\frac{4d}{Lb_dr_d\eps_P} \cdot \int_{-\infty}^{M_\phi + \log b_d}e^t\cdot 
\Leb_d(D_t)\;\mathsf{d}t\\
&\leq \frac{4d}{Lb_dr_d\eps_P} \cdot \int_{-\infty}^{M_\phi }e^t\cdot 
\Leb_d(D_t)\;\mathsf{d}t\\
&= \frac{4d}{Lb_dr_d\eps_P},
\end{align*}
where for the last step we again apply~\eqref{eqn:t_integral_no_inflate}. This completes the proof of Lemma~\ref{lem:Lip_maj}.

\subsubsection{Supporting lemmas for Lemma~\ref{lem:Lip_maj}}

\begin{lemma}\label{lem:sup}
For any $x\in\R^d$ and any $\phi\in \Phi_d$, there exists a Borel measurable map $x\mapsto y_x$
such that $y_x$ attains  $\sup_{y\in\R^d}\{\phi(y)-L\norm{x-y}\}$.
\end{lemma}
\begin{proof}[Proof of Lemma~\ref{lem:sup}]
Let $M_\phi:= \sup_{x \in \R^d} \phi(x)$, 
and let $x_{\phi}\in\argmax_{x\in\R^d}\phi(x)$
(note that, by definition of $\Phi_d\ni \phi$, $M_\phi$ must be finite, and $x_\phi$ must exist).
Define
\[\mathcal{Y}= \left\{y\in\R^d: \phi(y) \geq \phi(y') - L\norm{y-y'}\textnormal{ for all }y'\in\R^d\right\}.\]
Note that $\mathcal{Y}$ is non-empty, since trivially $x_\phi\in\mathcal{Y}$.

Next define $h:\R^d \times \mathcal{Y} \rightarrow \R$
as $h(x,y) = \phi(y)-L\norm{x-y}$. 
For each $x\in\R^d$, define
\[S(x) = \mathcal{Y}\cap\mathbb{B}_d(x, \norm{x-x_\phi}).\]
Note that, for any $x$, we have $x_\phi\in S(x)$ by definition.

Now we will apply \citet[Theorem 18.19]{aliprantis2006infinite}, which guarantees 
the existence of a Borel measurable function $x\mapsto y_x\in S(x)$ such that, for each $x$,
\[y_x \in\argmax_{y \in S(x)}h(x,y),\]
as long as we verify the following conditions:
\begin{itemize}
\item $\R^d$ is a measurable space, and $\mathcal{Y}$ is a separable metrizable space. This holds trivially.
\item $h$ is a Carath{\'e}odory function (i.e., $x\mapsto h(x,y)$ is measurable for any $y\in\mathcal{Y}$,
and $y\mapsto h(x,y)$ is continuous for almost every $x\in\R^d$). It holds trivially that $x\mapsto h(x,y)$ 
is measurable. To check that $y\mapsto h(x,y)$  is continuous for any fixed $x$, it is sufficient to verify that $\phi$ is continuous on $\mathcal{Y}$. 
In fact, examining the definition of $\mathcal{Y}$, we can see that $\phi$ is $L$-Lipschitz on $\mathcal{Y}$ by definition, thus
ensuring continuity.
\item $S(x)$ is non-empty and compact for any $x\in\R^d$. We have already seen that $x_\phi\in S(x)$ for all $x$.
To check compactness, it is sufficient to verify that $\mathcal{Y}$ is closed, which follows immediately from the definition of~$\mathcal{Y}$
along with the fact that $\phi$ is upper semi-continuous (by definition of $\phi\in\Phi_d$).
\item In the terminology of \citet{aliprantis2006infinite}, the correspondence $\mathcal{X} \twoheadrightarrow \mathcal{Y}$, mapping $x\mapsto S(x)\subseteq\mathcal{Y}$, is weakly measurable, meaning that 
the set $X_A := \{x\in\R^d : S(x)\cap A \neq \emptyset\}$ is measurable for any open subset $A\subseteq\mathcal{Y}$. \citet[Lemma 18.2]{aliprantis2006infinite}
establishes that, since $\mathcal{Y}$ is metrizable, this is implied by the stronger condition that $X_A$ is measurable
for every {\em closed} subset $A\subseteq\mathcal{Y}$, so we will check this stronger condition.

Let  $A\subseteq\mathcal{Y}$ be a closed subset. Consider any $x,x_1,x_2,\ldots\in\R^d$ such that $x_i\in X_A$ for all $i\geq 1$ and such that $\lim_{i\rightarrow\infty} x_i = x$.
Let $R = \sup_i \norm{x_i-x_\phi}$, which is finite
since the sequence converges. This means that $S(x_i)\subseteq \mathbb{B}_d(x_\phi,2R)$ for all $i$.
For each $i$, $x_i\in X_A$ implies that $S(x_i)\cap A \neq \emptyset$, and so we can find some $y_i \in S(x_i)\cap A\subseteq \mathbb{B}_d(x_\phi,2R)$.
Therefore, we can find some convergent subsequence, i.e., $i_1,i_2,\ldots$ such that $\lim_{j\rightarrow\infty} y_{i_j}=y$ for some $y\in\R^d$.
By assumption, $A$ is a closed subset of $\mathcal{Y}$, and we have already shown that $\mathcal{Y}$ is a closed subset
of $\R^d$. Therefore, $A\subseteq\R^d$ is closed, and so we must have $y\in A$. Now we check that $y\in S(x)$.
We know that $y\in A\subseteq \mathcal{Y}$, and so we only need to check that $y\in\mathbb{B}_d(x,\norm{x-x_\phi})$.
This holds because, for each $j\geq 1$, $y_{i_j}\in S(x_{i_j})\subseteq \mathbb{B}_d(x_{i_j},\norm{x_{i_j}-x_\phi})$, and so
\[\norm{y-x} = \lim_{j\rightarrow\infty} \norm{y_{i_j}-x_{i_j}} \leq  \lim_{j\rightarrow\infty} \norm{x_{i_j}-x_\phi} = \norm{x-x_\phi}.\]
We have now seen that $y\in S(x)\cap A$, proving that $S(x)\cap A\neq \emptyset$ and so $x\in X_A$. 
Therefore, we have established that $X_A$ is closed, and is therefore measurable.
\end{itemize}

\noindent Finally we check that, for any $x$,
\[\sup_{y\in\R^d}\{\phi(y)-L\norm{x-y}\} = \sup_{y\in S(x)}\{\phi(y)-L\norm{x-y}\}.\]
First, for any $y\not\in \mathbb{B}_d(x, \norm{x-x_\phi})$, we have
$\norm{x-y} > \norm{x-x_\phi}$, and so since $\phi(y)\leq \phi(x_\phi)$ by definition of $x_\phi$, it holds that
\[\phi(y)-L\norm{x-y}  <  \phi(x_\phi)-L\norm{x-x_\phi}.\] 
Therefore,
\[\sup_{y\in\R^d}\{\phi(y)-L\norm{x-y}\} =\sup_{y\in \mathbb{B}_d(x, \norm{x-x_\phi})}\{\phi(y)-L\norm{x-y}\}.\]
Next, since $\phi$ is upper semi-continuous, the supremum on the right-hand side is attained, i.e., 
there exists some $y_1\in \mathbb{B}_d(x, \norm{x-x_\phi})$ such that
\[\phi(y_1) - L\norm{x-y_1} = \sup_{y\in \mathbb{B}_d(x, \norm{x-x_\phi})}\{\phi(y)-L\norm{x-y}\} = \sup_{y\in\R^d}\{\phi(y)-L\norm{x-y}\} .\]
Now we verify that $y_1\in\mathcal{Y}$. To see this, fix any $y'\in\R^d$.
Then
\[\phi(y') - L\norm{x-y'} \leq  \sup_{y\in\R^d}\{\phi(y)-L\norm{x-y}\} = \phi(y_1) - L\norm{x-y_1}\]
and so
\[\phi(y_1) \geq \phi(y') - L\norm{x-y'} +  L\norm{x-y_1} \geq \phi(y') - L\norm{y_1-y'}.\]
Since this holds for all $y'\in\R^d$, we have established that $y_1\in\mathcal{Y}$. Therefore, $y_1\in S(x)$, which verifies
$\sup_{y\in\R^d}\{\phi(y)-L\norm{x-y}\} = \sup_{y\in S(x)}\{\phi(y)-L\norm{x-y}\}$. 
\end{proof}

\begin{lemma}\label{lem:convex_inflate}
Let $C\subseteq\R^d$ be any compact, convex set with non-empty interior. 
Then 
\[\delta\mapsto \frac{\Leb_d\big(\textnormal{Nbd}(C,\delta)\backslash C\big)}{\delta}\]
is a non-decreasing function of $\delta>0$.
\end{lemma}
\begin{proof}[Proof of Lemma~\ref{lem:convex_inflate}]
This result follows immediately from Steiner's formula \citep[Chapter 4]{schneider2014convex}, 
which states that for all $\eps\geq 0$,
\[\Leb_d\big(\textnormal{Nbd}(C,\eps)\big) = \Leb_d(C) + \sum_{k=1}^d V_{d-k}(C) \cdot \Leb_k(\mathbb{B}_k) \cdot \eps^k,\]
where  $V_{d-k}(C)\geq 0$ is the $(d-k)$-th intrinsic volume of $C$.
Rearranging, we have
\[\frac{\Leb_d\big(\textnormal{Nbd}(C,\eps)\backslash C\big)}{\eps} = \sum_{k=1}^d V_{d-k}(C) \cdot \Leb_k(\mathbb{B}_k) \cdot \eps^{k-1},\]
which is a non-decreasing function of $\epsilon$.
\end{proof}

\subsection{Proof of Lemma~\ref{lem:1d}}
First we consider the bounded case. Suppose that $P$ and $Q$ are both supported on $[-R,R]$ for some $R>0$.
Write $\Delta = \Delta_{\textnormal{CDF}}(P,Q)$ and $\eps = \min\{\eps_P,\eps_Q\}$. Let $r_1,b_1\in(0,1]$
be the universal constants defined in Lemma~\ref{lem:ball} (for dimension $d=1$), and fix any $L \geq \frac{4}{r_1\eps}$.
By Corollary~\ref{cor:Lip_maj}, we have
\[\dH^2\bigl(\psi^*(P),\psi^*(Q)\bigr) \leq \frac{16}{Lb_1r_1\eps}+ \left(\Ep{P}{\phi^L_P(X)} - \Ep{Q}{\phi^L_P(X)}\right) +  \left(\Ep{Q}{\phi^L_Q(X)}  - \Ep{P}{\phi^L_Q(X)}\right) .\]
Now we bound the two differences. 
For any $\phi\in\Phi_d$ define $M_{\phi} = \sup_{x\in\R^d}\phi(x)$ (note that $M_{\phi}$ is finite by definition of $\Phi_d$).
We note that $M_{\phi_P} = M_{\phi^L_P}$ by definition of $\phi^L_P$, and that $\phi^L_P(X)\geq M_{\phi_P}-2LR$ with probability $1$ under either $P$ or $Q$, since the distributions are supported on $[-R,R]$ and so
$\phi_P$ must attain its maximum somewhere in this range.
We then have
\begin{multline*}\Ep{P}{\phi^L_P(X)} - \Ep{Q}{\phi^L_P(X)}
=\Ep{Q}{M_{\phi_P} - \phi^L_P(X)} -   \Ep{P}{M_{\phi_P} - \phi^L_P(X)}\\
=\int_{0}^{2LR}  \left(\Pp{Q}{M_{\phi_P} - \phi^L_P(X)\geq t } - \Pp{P}{M_{\phi_P} - \phi^L_P(X)\geq t } \right)\;\mathsf{d}t.\end{multline*}
It is trivial to verify that
\[ \left|\sqrt{\Pp{P}{X\not\in C}} - \sqrt{\Pp{Q}{X\not\in C}} \right|\leq \Delta\sqrt{2}\]
for any convex set (i.e., an interval) $C\subseteq \R$, by definition of $\Delta$ (this follows from the fact that
$|\sqrt{a+c} - \sqrt{b+d}|^2 \leq |\sqrt{a} - \sqrt{b}|^2 + |\sqrt{c} - \sqrt{d}|^2$ for any $a,b,c,d\geq 0$).
Since $\phi^L_P$ is concave, the set $\{x:M_{\phi_P}-\phi^L_P(x)< t\}$ is convex, and so
\[\Pp{Q}{M_{\phi_P} - \phi^L_P(X)\geq t }  \leq \left(\sqrt{\Pp{P}{M_{\phi_P} - \phi^L_P(X)\geq t }} +\Delta\sqrt{2}\right)^2\]
and so, since it also holds that $\phi^L_P \geq \phi_P$ pointwise, we have
\[\Pp{Q}{M_{\phi_P} - \phi^L_P(X)\geq t } - \Pp{P}{M_{\phi_P} - \phi^L_P(X)\geq t }
\leq \Delta\sqrt{8}\cdot\sqrt{\Pp{P}{M_{\phi_P} - \phi_P(X)\geq t }}+2\Delta^2.\]
Lemma~\ref{lem:phi_with_radius_bound} below will establish that, for $t\geq \frac{8R}{r_1\eps}$, we have
$\Pp{P}{M_{\phi_P} - \phi_P(X)\geq t } \leq  \frac{32}{b_1r_1\eps} \cdot \frac{R}{t^2}$. Applying this bound,
 we have
\begin{align*}&\Ep{P}{\phi^L_P(X)} - \Ep{Q}{\phi^L_P(X)} \\
&\leq \int_{0}^{2LR} \left(\Delta\sqrt{8}\cdot\sqrt{\Pp{P}{M_{\phi_P} - \phi_P(X)\geq t }}+2\Delta^2 \right)\;\mathsf{d}t\\
&= \Delta\sqrt{8}\int_{0}^{2LR}\sqrt{\Pp{P}{M_{\phi_P} - \phi_P(X)\geq t }}\;\mathsf{d}t + 4LR\Delta^2\\
&=\Delta\sqrt{8}\biggl(\int_{0}^{\frac{8R}{r_1\eps}} \! \sqrt{\Pp{P}{M_{\phi_P} - \phi_P(X)\geq t }}\;\mathsf{d}t + 
\int_{\frac{8R}{r_1\eps}}^{2LR} \! \sqrt{\Pp{P}{M_{\phi_P} - \phi_P(X)\geq t }}\;\mathsf{d}t\biggr) +4LR\Delta^2\\
&\leq\Delta\sqrt{8}\sqrt{\frac{8R}{r_1\eps}}\cdot\biggl(\int_{0}^{\frac{8R}{r_1\eps}}\Pp{P}{M_{\phi_P} - \phi_P(X)\geq t }\;\mathsf{d}t \biggr)^{1/2} \! + 
\Delta\sqrt{8}\int_{\frac{8R}{r_1\eps}}^{2LR} \! \sqrt{ \frac{32}{b_1r_1\eps} \cdot \frac{R}{t^2}}\;\mathsf{d}t +4LR\Delta^2\\
&\leq \Delta\sqrt{8}\sqrt{\frac{8R}{r_1\eps}}\cdot\sqrt{\Ep{P}{M_{\phi_P}-\phi_P(X)}} + \Delta\sqrt{8}\sqrt{ \frac{32R}{b_1r_1\eps} } \log\left(Lr_1\eps/4\right) +4LR\Delta^2\\
&\leq \Delta\sqrt{8}\sqrt{\frac{8Rh_1}{r_1\eps}} + \Delta\sqrt{8}\sqrt{ \frac{32R}{b_1r_1\eps} } \log\left(Lr_1\eps/4\right)+ 4LR\Delta^2,\end{align*}
where the last step applies Lemma~\ref{lem:Ephi_const} below,
which will establish that $\Ep{P}{\phi(X)} \geq M_\phi - h_1$ for a universal constant $h_1$.
By symmetry the same bound holds for $\Ep{Q}{\phi^L_Q(X)} - \Ep{P}{\phi^L_Q(X)}$. Combining all our work so far, then,
\[\dH^2\bigl(\psi^*(P),\psi^*(Q)\bigr) \leq \frac{16}{Lb_1r_1\eps}+ 2\biggl\{ \Delta\sqrt{8}\biggl( \sqrt{\frac{8Rh_1}{r_1\eps}} + \sqrt{ \frac{32R}{b_1r_1\eps} } \log\left(Lr_1\eps/4\right)\biggr) + 4LR\Delta^2\biggr\}.\]
Next we split into cases. If $ \frac{1}{\Delta\sqrt{R\eps}} \geq \frac{4}{r_1\eps}$, then setting $L= \frac{1}{\Delta\sqrt{R\eps}}$ we apply this bound to obtain
\[\dH^2\bigl(\psi^*(P),\psi^*(Q)\bigr) \leq C' \Delta\sqrt{R/\eps} \max\biggl\{1,\log\Bigl(\frac{1}{ \Delta\sqrt{R/\eps}}\Bigr)\biggr\},\]
for a universal constant $C'$. Since $\eps\leq 2R$ by definition, and $\Delta \leq 1$, we can relax this to 
\[\dH^2\bigl(\psi^*(P),\psi^*(Q)\bigr) \leq C' \Delta\sqrt{R/\eps} \log(e/\Delta).\]
If instead $\frac{1}{\Delta\sqrt{R\eps}} < \frac{4}{r_1\eps}$, then 
\[\dH^2\bigl(\psi^*(P),\psi^*(Q)\bigr) \leq 2 \leq \frac{8}{r_1} \Delta\sqrt{R/\eps}.\]
Therefore, combining both cases, we have
\begin{equation}\label{eqn:bounded_support}\dH^2\bigl(\psi^*(P),\psi^*(Q)\bigr) \leq C''  \Delta\sqrt{\frac{R}{\min\{\eps_P,\eps_Q\}}}\log(e/\Delta)\end{equation}
for a universal constant $C'' =\max\{C',8/r_1\}$.  Next we will need to relate $\min\{\eps_P,\eps_Q\}$ with $\max\{\eps_P,\eps_Q\}$. 
Without loss of generality, suppose that $\mu_P\geq \mu_Q$. We then have
\begin{align*}
\frac{\eps_Q}{2}
&=\frac{1}{2}\Ep{{Q}}{|X-\mu_{Q}|}
=\Ep{{Q}}{(X-\mu_{Q})_+}\\
&\geq \Ep{{Q}}{(X-\mu_{P})_+}
=\int_{\mu_{P}}^{R}\Pp{{Q}}{X>t}\;\mathsf{d}t\\
&\geq \int_{\mu_{P}}^{R}\Pp{{P}}{X>t} - 2\Delta\sqrt{\Pp{{P}}{X>t}} \;\mathsf{d}t\\
&\geq \int_{\mu_{P}}^{R}\Pp{{P}}{X>t} \;\mathsf{d}t - 2\Delta\sqrt{R-\mu_{P}}\sqrt{\int_{\mu_{P}}^R\Pp{{P}}{X>t} \;\mathsf{d}t}\\
  &\geq\Ep{{P}}{(X-\mu_{P})_+} - 2\Delta\sqrt{2R}\sqrt{\Ep{{P}}{(X-\mu_{P})_+}} \\
&=\frac{\eps_P}{2} - 2\Delta\sqrt{R\cdot \eps_P},
\end{align*}
where the final inequality follows because $|\mu_{P}|\leq R$.  We can similarly calculate
\[\frac{\eps_P}{2} = \frac{1}{2}\Ep{{P}}{|X-\mu_{P}|} = \Ep{{P}}{(X-\mu_{P})_-} \geq \frac{\eps_Q}{2} - 2\Delta\sqrt{R\cdot \eps_Q}.\]
Combining these two bounds, then,
\begin{equation}\label{eqn:epsP_epsQ_bounded_support}\max\{\eps_P,\eps_Q\} = \min\{\eps_P,\eps_Q\} + |\eps_P-\eps_Q|\leq \min\{\eps_P,\eps_Q\}  + 4\Delta_{\textnormal{CDF}}(P,Q)\cdot\sqrt{R\cdot \max\{\eps_P,\eps_Q\}}.\end{equation}

Now we  work with the general case, where $P,Q$ may not have bounded support. Fix any $R>0$.
For any $x\in\R$ define
\begin{equation}\label{eqn:x_R_def}[x]_R := \begin{cases}-R, & x<-R,\\ x, & |x|\leq R, \\ R, & x>R,\end{cases}\end{equation}
the truncation of $x$ to the range $[-R,R]$. Let $[P]_R$ denote the distribution of $[X]_R$ when $X\sim P$, and same for $[Q]_R$. Lemma~\ref{lem:dw_bound} below calculates that $\dw(P,[P]_R) \leq  \frac{M_q^q}{R^{q-1}}$.
Applying Theorem~\ref{thm:main} to compare the distributions $P$ and $[P]_R$, then, we have
\[\dH^2\bigl(\psi^*(P),\psi^*([P]_R)\bigr) \leq C_1^2\sqrt{\frac{\dw(P,[P]_R)}{\max\{\eps_P,\eps_{[P]_R}\}}} \leq C_1^2\sqrt{\frac{M_q^q}{\eps_{[P]_R} R^{q-1}}},\]
and the same bound holds with $Q$ in place of $P$.
Therefore, by the triangle inequality,
\begin{align}
\notag\dH^2\bigl(\psi^*(P),&\psi^*(Q)\bigr) \\
\notag&\leq \Bigl\{\dH\bigl(\psi^*(P),\psi^*([P]_R)\bigr)+\dH\bigl(\psi^*(Q),\psi^*([Q]_R)\bigr)+\dH\bigl(\psi^*([P]_R),\psi^*([Q]_R)\bigr)\Bigr\}^2\\
\notag&\leq 3\dH^2\bigl(\psi^*(P),\psi^*([P]_R)\bigr) + 3\dH^2\bigl(\psi^*(Q),\psi^*([Q]_R)\bigr)+3\dH^2\bigl(\psi^*([P]_R),\psi^*([Q]_R)\bigr)\\
\label{eqn:unbounded_to_bounded}&\leq 6C_1^2\sqrt{\frac{M_q^q}{\min\{\eps_{[P]_R},\eps_{[Q]_R}\} R^{q-1}}}+3\dH^2\bigl(\psi^*([P]_R),\psi^*([Q]_R)\bigr).\end{align}
We now need to apply the bound~\eqref{eqn:bounded_support} to the bounded distributions $[P]_R$ and $[Q]_R$,
in order to bound this last term. Combining~\eqref{eqn:bounded_support} with~\eqref{eqn:unbounded_to_bounded}, we obtain
\begin{multline*}\dH^2\bigl(\psi^*(P),\psi^*(Q)\bigr)
\leq 6C_1^2\sqrt{\frac{M_q^q}{\min\{\eps_{[P]_R},\eps_{[Q]_R}\} R^{q-1}}}\\{}+3C'' \Delta_{\textnormal{CDF}}([P]_R,[Q]_R)\sqrt{\frac{R}{\min\{\eps_{[P]_R},\eps_{[Q]_R}\}}}\log\bigl(e/\Delta_{\textnormal{CDF}}([P]_R,[Q]_R)\bigr).\end{multline*}
Now fix 
\[R=M_q\Big\{\Delta_{\textnormal{CDF}}([P]_R,[Q]_R)\log(e/\Delta_{\textnormal{CDF}}([P]_R,[Q]_R))\Big\}^{-2/q}.\] This yields
\begin{align*}\dH^2\bigl(\psi^*(P)&,\psi^*(Q)\bigr) \\
&\leq C_*'\sqrt{\frac{M_q}{\min\{\eps_{[P]_R},\eps_{[Q]_R}\}}}\cdot \bigg\{\Delta_{\textnormal{CDF}}([P]_R,[Q]_R)\log\biggl(\frac{e}{\Delta_{\textnormal{CDF}}([P]_R,[Q]_R)}\biggr)\bigg\}^{1-1/q},\end{align*}
when the universal constant $C_*' > 0$ is chosen appropriately. Next, it holds trivially that $\Delta_{\textnormal{CDF}}([P]_R,[Q]_R)\leq\Delta_{\textnormal{CDF}}(P,Q) $, and since $t\mapsto t\log(e/t)$ is increasing on $t\in(0,1]$, we therefore have
\[\dH^2\bigl(\psi^*(P),\psi^*(Q)\bigr)
\leq C_*'\sqrt{\frac{M_q}{\min\{\eps_{[P]_R},\eps_{[Q]_R}\}}}\cdot \Big\{\Delta_{\textnormal{CDF}}(P,Q)\log\bigl(e/\Delta_{\textnormal{CDF}}(P,Q)\bigr)\Big\}^{1-1/q}.\]

Finally, we need to lower bound $\eps_{[P]_R}$ and $\eps_{[Q]_R}$. 
First we relate $\min\{\eps_{[P]_R},\eps_{[Q]_R}\}$ to $\max\{\eps_{[P]_R},\eps_{[Q]_R}\}$. 
Applying~\eqref{eqn:epsP_epsQ_bounded_support} from above, along with the fact that $\Delta_{\textnormal{CDF}}([P]_R,[Q]_R)\leq\Delta_{\textnormal{CDF}}(P,Q) $,
we  have
\[\max\{\eps_{[P]_R},\eps_{[Q]_R}\} \leq \min\{\eps_{[P]_R},\eps_{[Q]_R}\} + 4\Delta_{\textnormal{CDF}}(P,Q)\sqrt{R\cdot \max\{\eps_{[P]_R},\eps_{[Q]_R}\}}.\]
If $8\Delta_{\textnormal{CDF}}(P,Q)\sqrt{R} \leq \sqrt{\max\{\eps_{[P]_R},\eps_{[Q]_R}\}}$, then this proves that 
\[\max\{\eps_{[P]_R},\eps_{[Q]_R}\} \leq 2\min\{\eps_{[P]_R},\eps_{[Q]_R}\} \] and so
\[\dH^2\bigl(\psi^*(P),\psi^*(Q)\bigr)
\leq C_*'\sqrt{\frac{2M_q}{\max\{\eps_{[P]_R},\eps_{[Q]_R}\}}}\cdot \Big\{\Delta_{\textnormal{CDF}}(P,Q)\log\bigl(e/\Delta_{\textnormal{CDF}}(P,Q)\bigr)\Big\}^{1-1/q}.\]
If instead $8\Delta_{\textnormal{CDF}}(P,Q)\sqrt{R} > \sqrt{\max\{\eps_{[P]_R},\eps_{[Q]_R}\}}$, then we have
\[\dH^2\bigl(\psi^*(P),\psi^*(Q)\bigr) \leq 2 \leq \frac{16\Delta_{\textnormal{CDF}}(P,Q)\sqrt{R}}{\sqrt{\max\{\eps_{[P]_R},\eps_{[Q]_R}\}}}.\]
Plugging in the definition of $R$ and combining both cases, we obtain
\[\dH^2\bigl(\psi^*(P),\psi^*(Q)\bigr)
\leq C_*''\sqrt{\frac{M_q}{\max\{\eps_{[P]_R},\eps_{[Q]_R}\}}}\cdot \Big(\Delta_{\textnormal{CDF}}(P,Q)\log(e/\Delta_{\textnormal{CDF}}(P,Q))\Big)^{1-1/q}\]
for an appropriately chosen universal constant $C_*''$.
The last step is to relate $\max\{\eps_{[P]_R},\eps_{[Q]_R}\}$ to $\max\{\eps_P,\eps_Q\}$.
 Applying Proposition~\ref{prop:eps_P} together
with the bound on $\dw(P,[P]_R)$ from Lemma~\ref{lem:dw_bound}, we have
\[\eps_{[P]_R}\geq \eps_P - 2\dw(P,[P]_R) \geq\eps_P - 2\cdot \frac{M_q^q}{R^{q-1}},\]
and the same bound holds for $Q$ in place of $P$.
If $\frac{2M_q^q}{ R^{q-1}} \leq \frac{\max\{\eps_P,\eps_Q\}}{2}$,
then \[\max\{\eps_{[P]_R},\eps_{[Q]_R}\}\geq \frac{\max\{\eps_P,\eps_Q\}}{2},\] and so we obtain
\[\dH^2\bigl(\psi^*(P),\psi^*(Q)\bigr)
\leq C_*''\sqrt{\frac{2M_q}{\max\{\eps_P,\eps_Q\}}}\cdot \Big\{\Delta_{\textnormal{CDF}}(P,Q)\log\bigl(e/\Delta_{\textnormal{CDF}}(P,Q)\bigr)\Big\}^{1-1/q}.\]
If instead $\frac{2M_q^q}{ R^{q-1}} > \frac{\max\{\eps_P,\eps_Q\}}{2}$, then it trivially holds that
\[\dH^2\bigl(\psi^*(P),\psi^*(Q)\bigr)\leq 2  
\leq 2\sqrt{\frac{4M_q^q}{\max\{\eps_P,\eps_Q\} R^{q-1}}}.\]
Plugging in the definition of $R$, and combining the two cases, we obtain
\[\dH^2\bigl(\psi^*(P),\psi^*(Q)\bigr)
\leq C_*\sqrt{\frac{M_q}{\max\{\eps_P,\eps_Q\}}}\cdot \Big\{\Delta_{\textnormal{CDF}}(P,Q)\log\bigl(e/\Delta_{\textnormal{CDF}}(P,Q)\bigr)\Big\}^{1-1/q}\]
for appropriately chosen universal constant $C_*$, which completes the proof of Lemma~\ref{lem:1d}.
\subsubsection{Supporting lemmas for Lemma~\ref{lem:1d}}

\begin{lemma}\label{lem:phi_with_radius_bound}
Let $P\in\Pcal_d$ and let $\phi = \phi^*(P)$. Let $M_{\phi} := \sup_{x\in\R^d}\phi(x)$ and let $x_{\phi}\in\argmax_{x\in\R^d}\phi(x)$ 
(which is guaranteed to exist by definition of $\Phi_d\ni \phi$).
Fix any $R>0$ and $t\geq \frac{8dR}{r_d\eps_P}$, where $r_d\in(0,1]$ is taken from Lemma~\ref{lem:ball}. Then
\[\Pp{P}{\phi(X) \leq M_{\phi} - t \textnormal{ and }\norm{X-x_{\phi}}\leq 2R}\leq \frac{32d}{b_dr_d\eps_P} \cdot \frac{R}{t^2},\]
where $b_d\in(0,1]$ is taken from Lemma~\ref{lem:ball}.
\end{lemma}
\begin{proof}[Proof of Lemma~\ref{lem:phi_with_radius_bound}]
First, for any $x$ with $\norm{x-x_{\phi}}\leq 2R$, 
\[\phi^{t/4R}(x) = \sup_{y\in\R^d}\biggl\{\phi(y) - \frac{t}{4R}\norm{y-x}\biggr\} \geq \phi(x_{\phi}) - \frac{t}{4R}\norm{x-x_\phi} \geq M_\phi - \frac{t}{2}.\]
Hence, if $\phi(x) \leq M_\phi - t$ and  $\norm{x-x_{\phi}}\leq 2R$, then
\[\phi^{t/4R}(x) - \phi(x) \geq \frac{t}{2}.\]
Moreover, by definition of $\phi = \phi^*(P)$, since $\phi^{t/4R}\in\Phi_d$, it holds that
\begin{align*}\Ep{P}{\phi(X)} = \ell(\phi,P) \geq \ell(\phi^{t/4R},P) &= \Ep{P}{\phi^{t/4R}(X)}  -\int_{\R^d}e^{\phi^{t/4R}(x)}\;\mathsf{d}x + 1\\ &\geq \Ep{P}{\phi^{t/4R}(X)}  - \frac{4d}{\frac{t}{4R}b_dr_d\eps_P} ,\end{align*}
where the last step holds by~\eqref{eqn:Lip_ball_inflate} as calculated in the proof of Lemma~\ref{lem:Lip_maj}, noting that $\frac{t}{4R} \geq \frac{2d}{r_d\eps_P}$. We deduce that
\begin{multline*}
\Pp{P}{\phi(X) \leq M_\phi - t \textnormal{ and }\norm{X-x_\phi}\leq 2R}
\leq \Pp{P}{\phi^{t/4R}(X) - \phi(X) \geq \frac{t}{2}}\\
\leq \frac{\Ep{P}{\phi^{t/4R}(X) - \phi(X)}}{t/2}
\leq \frac{\frac{4d}{\frac{t}{4R}b_dr_d\eps_P}}{t/2}
= \frac{32d}{b_dr_d\eps_P} \cdot \frac{R}{t^2},
\end{multline*}
as required.
\end{proof}

\begin{lemma}\label{lem:Ephi_const}
Fix any $P\in\Pcal_d$ and let $\phi = \phi^*(P)$. Then
\[\Ep{P}{\phi(X)} \geq M_\phi - h_d,\]
where $M_\phi=\sup_{x\in\R^d}\phi(x)$ and where $h_d \geq 0$ depends only on $d$.
\end{lemma}

\begin{proof}[Proof of Lemma~\ref{lem:Ephi_const}]
Write $\Ep{\phi}{\cdot}$ to denote the expectation with respect to the 
distribution with log-density $\phi$. Let $\mu_\phi := \Ep{\phi}{X}$ be the mean 
and $\Sigma := \Ep{\phi}{(X-\mu_\phi)(X-\mu_\phi)^\top}$ the covariance of this distribution.
Let $\bar{\phi}$ denote the log-density of the isotropic, log-concave random vector $\Sigma^{-1/2}(X-\mu_\phi)$, where $X$ has log-density $\phi$.  Let $M_{\bar\phi} := \sup_{x\in\R^d}\bar\phi(x)$.

Since $x\mapsto \phi(x) + \frac{1}{2}\bigl\{M_\phi - \phi(x)\bigr\}$ is concave and coercive, it holds by  \citet[Remark 2.3]{dumbgen2011approximation} that
\[\Ep{P}{M_\phi -\phi(X)} \leq \Ep{\phi}{M_\phi -\phi(X)}.\]
Next, we can trivially verify that 
\[\Ep{\phi}{M_\phi -\phi(X)} = \Ep{\bar\phi}{M_{\bar\phi} -\bar\phi(X)}\]
since the log-densities $\phi$ and $\bar\phi$ are related via the  linear transformation on random variables above.
Furthermore,
\[\Ep{\bar\phi}{M_{\bar\phi} -\bar\phi(X)} = M_{\bar\phi} - \int_{\R^d} e^{\bar\phi(y) } \cdot\bar\phi(y)\;\mathsf{d}y \leq M_{\bar\phi} + \frac{d}{2}\log(2\pi e),
\]
where the last step holds since $\bar\phi$ is the log-density of an isotropic distribution on $\R^d$, and so its entropy 
is bounded by that of the standard $d$-dimensional Gaussian \citep[e.g.][Theorem~9.6.5]{cover1991elements}. Finally, by  \citet[Theorem 5.14(e)]{lovasz2007geometry}, $M_{\bar\phi}\leq m_d$ where $m_d \in \mathbb{R}$ depends only on the dimension $d$.
Therefore, combining everything, 
\[\Ep{P}{M_\phi -\phi(X)} \leq  m_d +  \frac{d}{2}\log(2\pi e),\]
which proves the desired bound.
\end{proof}

\begin{lemma}\label{lem:dw_bound}
Let $P\in\Pcal_1$ satisfy $\Ep{P}{|X|^q}^{1/q}\leq M_q$, for some $q>1$.
Let $[P]_R$ be the distribution of $[X]_R$ when $X\sim P$ (where the truncation $[X]_R$ is defined as in~\eqref{eqn:x_R_def}).
Then
\[\dw(P,[P]_R) \leq  \frac{M_q^q}{R^{q-1}}.\]
\end{lemma}

\begin{proof}[Proof of Lemma~\ref{lem:dw_bound}]
Drawing $X\sim P$, note that $(X,[X]_R)$ is a coupling of the distributions $P$ and $[P]_R$. Hence 
\[\dw(P,[P]_R) \leq \Ep{P}{|X-[X]_R|} = \Ep{P}{\big(|X|-R\big)_+} \leq \Ep{P}{\frac{|X|^q}{R^{q-1}}} \leq \frac{M_q^q}{R^{q-1}},\]
as required.
\end{proof}

\subsection{Proof of Lemma~\ref{lem:cdf}}
Write $\widehat{U}_n(t) = \frac{1}{n}\sum_{i=1}^n\One{U_i\leq t}$. First we calculate
\[\Delta_{\textnormal{CDF}}\bigl(\widehat{U}_n,\textnormal{Unif}[0,1]\bigr)
=\max\left\{ \underbrace{\sup_{t\in[0,1]}\left|\sqrt{1-\widehat{U}_n(t)}-\sqrt{1-t}\right| }_{=\Delta_0} \ \ ,\ \ \underbrace{ \sup_{t\in[0,1]}\left|\sqrt{\widehat{U}_n(t)}-\sqrt{t}\right|}_{=:\Delta_1}\right\},\]
by obseving that
\[\sup_{t\in[0,1]}\left|\sqrt{\frac{1}{n}\sum_{i=1}^n\One{U_i<t}}-\sqrt{t}\right| = \sup_{t\in[0,1]}\left|\sqrt{\frac{1}{n}\sum_{i=1}^n\One{U_i\leq t}}-\sqrt{t}\right|\]
(i.e., the supremum is unchanged by replacing $<$ with $\leq$).
We can further write
\[\Delta_1 = \max\left\{\underbrace{\sup_{t\in[0,\frac{\log n}{n}]}\left|\sqrt{\widehat{U}_n(t)}-\sqrt{t}\right|}_{=:\Delta_{1,0}} \ \ , \ \
\underbrace{\sup_{t\in[\frac{\log n}{n},1-\frac{\log n}{n}]}\left|\sqrt{\widehat{U}_n(t)}-\sqrt{t}\right|}_{=:\Delta_{1,1}} \ \ , \ \
\underbrace{\sup_{t\in[1-\frac{\log n}{n},1]}\left|\sqrt{\widehat{U}_n(t)}-\sqrt{t}\right|}_{=:\Delta_{1,2}}\right\}.\]
We have
\[\Delta_{1,0} = \sup_{t\in[0,\frac{\log n}{n}]}\left|\sqrt{\widehat{U}_n(t)}-\sqrt{t}\right|
\leq \sqrt{\frac{\log n}{n}} + \sqrt{\widehat{U}_n\left(\frac{\log n}{n}\right)}\leq 2\sqrt{\frac{\log n}{n}}  + \Delta_{1,1},\]
and 
\[\Delta_{1,2} = \sup_{t\in[1-\frac{\log n}{n},1]}\left|\sqrt{\widehat{U}_n(t)}-\sqrt{t}\right| \leq 
\sqrt{\frac{\log n}{n}} +\left( 1-  \sqrt{\widehat{U}_n\left(1 - \frac{\log n}{n}\right)}\right)\leq 2\sqrt{\frac{\log n}{n}}  + \Delta_{1,1}.\]
Furthermore,
\[\Delta_{1,1} = \sup_{t\in[\frac{\log n}{n},1-\frac{\log n}{n}]}\left|\sqrt{\widehat{U}_n(t)}-\sqrt{t}\right|
=\sup_{t\in[\frac{\log n}{n},1-\frac{\log n}{n}]}\frac{|{\widehat{U}_n(t)}-{t}|}{\sqrt{\widehat{U}_n(t)}+\sqrt{t}} \leq 
\sup_{t\in[\frac{\log n}{n},1-\frac{\log n}{n}]}\frac{|{\widehat{U}_n(t)}-{t}|}{\sqrt{t}}.\]
Combining these calculations, we have
\[\Delta_1 \leq 2\sqrt{\frac{\log n}{n}}  + \sup_{t\in[\frac{\log n}{n},1-\frac{\log n}{n}]}\frac{|{\widehat{U}_n(t)}-{t}|}{\sqrt{t}}.\]
Similarly we can calculate
\[\Delta_0 \leq 2\sqrt{\frac{\log n}{n}}  + \sup_{t\in[\frac{\log n}{n},1-\frac{\log n}{n}]}\frac{|{\widehat{U}_n(t)}-{t}|}{\sqrt{1-t}},\]
and so we have
\begin{multline*}\Delta_{\textnormal{CDF}}\bigl(\widehat{U}_n,\textnormal{Unif}[0,1]\bigr) \leq  2\sqrt{\frac{\log n}{n}}  +
\sup_{t\in[\frac{\log n}{n},1-\frac{\log n}{n}]}\frac{|{\widehat{U}_n(t)}-{t}|}{\sqrt{\min\{t,1-t\}}}\\
=2\sqrt{\frac{\log n}{n}}  +
\max\left\{\sup_{t\in[\frac{\log n}{n},\frac{1}{2}]}\frac{|{\widehat{U}_n(t)}-{t}|}{\sqrt{t}},
\sup_{t\in[\frac{1}{2},1-\frac{\log n}{n}]}\frac{|{\widehat{U}_n(t)}-{t}|}{\sqrt{1-t}}\right\}.\end{multline*}

Next, \citet[Proposition 11.1.1 (part (10)) + Inequality 11.2.1]{shorack2009empirical} (applied with $q(t)=\sqrt{t}$, with $a=\frac{\log n}{n}$, and with $b=\delta=\frac{1}{2}$) establishes that, for any $\lambda>0$,
\[\PP{\sup_{t\in[\frac{\log n}{n},\frac{1}{2}]}\frac{|{\widehat{U}_n(t)}-{t}|}{\sqrt{t}}\geq \frac{\lambda}{\sqrt{n}}}\leq 12 \int_{\frac{\log n}{n}}^{1/2} \frac{1}{t} \cdot \exp\left\{-\frac{\lambda^2}{8\left(1+\frac{\lambda}{3\sqrt{\log n}}\right)}\right\}\;\mathsf{d}t,\]
as long as $n$ satisfies $\frac{\log n}{n}\leq\frac{1}{4}$ (which holds for $n>8$; for $n\leq 8$, by taking $c'\geq 2$ we can ensure that the lemma's claim is trivial, since $\Delta_{\textnormal{CDF}}\bigl(\widehat{U}_n,\textnormal{Unif}[0,1]\bigr)\leq 1$ deterministically).
Furthermore, clearly we see that $\sup_{t\in[\frac{\log n}{n},\frac{1}{2}]}\frac{|{\widehat{U}_n(t)}-{t}|}{\sqrt{t}}$ and $\sup_{t\in[\frac{1}{2},1-\frac{\log n}{n}]}\frac{|{\widehat{U}_n(t)}-{t}|}{\sqrt{1-t}}$ are equal in distribution. Therefore, we have
\[\PP{\Delta_{\textnormal{CDF}}\bigl(\widehat{U}_n,\textnormal{Unif}[0,1]\bigr) \geq 2\sqrt{\frac{\log n}{n}}+\frac{\lambda}{\sqrt{n}}} \leq 24\log\left(\frac{n}{2\log n}\right)\cdot \exp\left\{-\frac{\lambda^2}{8\left(1+\frac{\lambda}{3\sqrt{\log n}}\right)}\right\}\]
for any $\lambda>0$. 
Taking $\lambda = 5(c+2)\sqrt{\log n}$, we can calculate $\exp\left\{-\frac{\lambda^2}{8\left(1+\frac{\lambda}{3\sqrt{\log n}}\right)}\right\} \leq
\exp\{-(c+2)\log n\} = n^{-(c+2)}$,
and so we have
\[\PP{\Delta_{\textnormal{CDF}}\bigl(\widehat{U}_n,\textnormal{Unif}[0,1]\bigr) \geq 2\sqrt{\frac{\log n}{n}}+5(c+2)\sqrt{\frac{\log n}{n}}}
\leq 24\log\left(\frac{n}{2\log n}\right)\cdot  n^{-(c+2)} \leq n^{-c}\]
where the last step holds since we have assumed that $n>8$. This proves the lemma with $c' = 5c+12$.

\subsection{Proof of Lemma~\ref{lem:prob_above_below_mu}}
We have
\begin{align*}
\eps_P &= \Ep{P}{|X-\mu_P|} \\
&=  2\Ep{P}{(X-\mu_P)_+} \\
&\leq 2\Ep{P}{|X-\mu_P| \cdot \One{X>\mu_P}} \\
&\leq 2\Ep{P}{|X-\mu_P|^q}^{1/q}\Ep{P}{\One{X>\mu_P}^{\frac{q}{q-1}}}^{\frac{q-1}{q}}\\
&\leq 2(\Ep{P}{|X|^q}^{1/q} + (|\mu_P|^q)^{1/q}) \cdot \Pp{P}{X>\mu_P}^{\frac{q-1}{q}}\\
&\leq 4M_q \cdot \Pp{P}{X>\mu_P}^{\frac{q-1}{q}}.\end{align*}
Therefore,
\[\Pp{P}{X>\mu_P} \geq \left(\frac{\eps_P}{4M_q}\right)^{\frac{q}{q-1}}.\]
Similarly, the same bound holds for $\Pp{P}{X<\mu_P}$.

\subsection{Proofs of lower bounds (Theorems~\ref{thm:lowerbd} and~\ref{thm:lowerbd_empirical})}
We begin with some preliminary calculations that we will use for the constructions for both theorems.  Fix any $0<\rho_0<\rho_1$ and any $\beta\in\bigl(0,\rho_0/\rho_1\bigr]$. Let $P$ be the mixture distribution drawing 
\begin{equation}\label{eqn:lowerbd_construction_distribution}
X\sim \begin{cases}\textnormal{Unif}\bigl(\mathbb{S}_{d-1}(\rho_0)\bigr),& \textnormal{ with probability $1-\beta$,}\\
\textnormal{Unif}\bigl(\mathbb{S}_{d-1}(\rho_1)\bigr),&\textnormal{ with probability $\beta$.}\end{cases}\end{equation}
Defining
\begin{equation}\label{eqn:def_s_d}s_d= \EE{|V_1|} \textnormal{ for $V = (V_1,\ldots,V_d) \sim\textnormal{Unif}(\mathbb{S}_{d-1})$},\end{equation} 
we can calculate 
\[\eps_P = (1-\beta)\rho_0 \cdot s_d + \beta\rho_1\cdot s_d \geq s_d\rho_0.\]
We will apply Lemma~\ref{lem:phi_with_radius_bound} to this distribution $P$ and the log-density $\phi=\phi^*(P)$ of its log-concave projection.
Observe that $\phi$ is spherically symmetric around 0, and is constant over $\norm{x}\leq \rho_0$---in particular, this means that $\phi(x)=M_\phi$ for all
$\norm{x}\leq \rho_0$, where $M_\phi = \sup_{x\in\R^d}\phi(x)$ as before.
Next, let  $t_* \geq 0$ be the value of $M_\phi - \phi(x)$ for points $x$ with $\norm{x}=\rho_1$ (since $\phi$ is spherically symmetric, this is well defined).
We now split into cases. If $t_*\geq \frac{8d\rho_1}{r_ds_d\rho_0}$, 
then applying Lemma~\ref{lem:phi_with_radius_bound} with $R=\rho_1/2$, $x_\phi=0$, and $t=t_*$, we obtain
\[\beta\leq \Pp{P}{\phi(X) \leq M_\phi - t_* \textnormal{ and }\norm{X}\leq \rho_1}\leq \frac{16d}{b_dr_ds_d\rho_0} \cdot \frac{\rho_1}{t_*^2},\]
which proves that
\[t_*\leq\sqrt{\frac{16d}{b_dr_ds_d}\cdot \frac{\rho_1}{\rho_0\beta}}.\]
If this case does not hold, then we instead have  
$t_*< \frac{8d\rho_1}{r_ds_d\rho_0}$, so combining the two cases,
\[t_*\leq\max\left\{\sqrt{\frac{16d}{b_dr_ds_d}\cdot \frac{\rho_1}{\rho_0\beta}},\frac{8d}{r_ds_d}\cdot \frac{\rho_1}{\rho_0}\right\}
\leq\max\left\{\sqrt{\frac{16d}{b_dr_ds_d}},\frac{8d}{r_ds_d}\right\}\cdot\sqrt{ \frac{\rho_1}{\rho_0\beta}},\]
where the last step comes from our assumption on $\beta$.
Therefore
\[\phi(x) \geq \phi(0) - \max\left\{\sqrt{\frac{16d}{b_dr_ds_d}},\frac{8d}{r_ds_d}\right\}\cdot\sqrt{ \frac{\rho_1}{\rho_0\beta}}\]
for $\norm{x} = \rho_1$ while
\[\phi(x) = \phi(0)\]
for $\norm{x}\leq\rho_0$. By concavity of $\phi$, then,
\[\phi(x) \geq \phi(0) - \max\left\{\sqrt{\frac{16d}{b_dr_ds_d}},\frac{8d}{r_ds_d}\right\} \]
for all $x$ with
$\norm{x}\leq \rho_0 + (\rho_1-\rho_0) \cdot \sqrt{\frac{\rho_0\beta}{\rho_1}}$.
Therefore, for any density $f$ supported on $\mathbb{B}_d(\rho_0)$, it holds that
\begin{align*}
\dH^2\bigl(f,\psi^*(P)\bigr)
& \geq \int_{\R^d} e^{\phi(0) - \max\left\{\sqrt{\frac{16d}{b_dr_ds_d}},\frac{8d}{r_ds_d}\right\} }\cdot\One{\rho_0 < \norm{x} < \rho_0 + (\rho_1-\rho_0) \cdot \sqrt{\frac{\rho_0\beta}{\rho_1}}}\;\mathsf{d}x\\
&=e^{\phi(0) - \max\left\{\sqrt{\frac{16d}{b_dr_ds_d}},\frac{8d}{r_ds_d}\right\} }\cdot \Leb_d\Big(\mathbb{B}_d\big(\rho_0 + (\rho_1-\rho_0) \cdot \sqrt{\rho_0\beta/\rho_1}\big) \backslash \mathbb{B}_d(\rho_0)\Big)\\
&\geq e^{\phi(0) - \max\left\{\sqrt{\frac{16d}{b_dr_ds_d}},\frac{8d}{r_ds_d}\right\} }\cdot \rho_0^{d-1} \cdot (\rho_1-\rho_0)\cdot \sqrt{\frac{\rho_0\beta}{\rho_1}}\cdot  S_{d-1},
\end{align*}
where as before $S_{d-1}$ denotes the surface area of $\mathbb{S}_{d-1}$. Finally, we need to place a lower bound on $\phi(0)$. By Corollary~\ref{cor:lambda_min}, we know that the covariance matrix $\Sigma$ of the 
distribution with log-density $\phi$ has operator norm bounded as 
\[\norm{\Sigma}_{{\mathrm{op}}} \leq 16\big((1-\beta)\rho_0 + \beta\rho_1\big)^2.\]
Furthermore, $\tilde\phi(x) = \frac{1}{2}\log\det(\Sigma)+ \phi(\Sigma^{1/2}x)$
is an isotropic concave log-density, and so $\tilde\phi(0)\geq c'_d$ where $c'_d>0$ depends only on $d$, by \citet[Theorem 5.14(d)]{lovasz2007geometry}.
Therefore, 
\[\phi(0)\geq c'_d - \frac{d}{2}\log(16) - d\log\big((1-\beta)\rho_0 + \beta\rho_1\big).\]We conclude that
\begin{equation}\label{eqn:lowerbd_construction_dH_step}\dH^2\bigl(f,\psi^*(P)\bigr) \geq c''_d \cdot  \rho_0^{d-1} \cdot (\rho_1-\rho_0)\cdot \sqrt{\frac{\rho_0\beta}{\rho_1}} \cdot \big((1-\beta)\rho_0 + \beta\rho_1\big)^{-d},\end{equation}
where $c''_d$ depends only on $d$.

\subsubsection{Completing the proof of Theorem~\ref{thm:lowerbd}}
To prove Theorem~\ref{thm:lowerbd}, let $P$ be the distribution constructed in~\eqref{eqn:lowerbd_construction_distribution} with 
\[\rho_0 = \eps/s_d, \ \rho_1 = 2\eps/s_d, \ \beta = \min\left\{\frac{s_d\delta}{\eps},\frac{1}{2}\right\},\]
where $s_d$ is defined as in~\eqref{eqn:def_s_d}. Let
\[Q = \textnormal{Unif}(\mathbb{S}_{d-1}\bigl(\rho_0)\bigr).\]
Clearly $\eps_P \geq \eps_Q = s_d\rho_0 = \eps$, and $\dw(P,Q) = \beta(\rho_1-\rho_0) \leq \delta$, thus satisfying the conditions of the theorem. 
Since $Q$ is supported on $\mathbb{B}_d(\rho_0)$, $\psi^*(Q)$ is also supported on this ball.
Then applying our calculation~\eqref{eqn:lowerbd_construction_dH_step}, and plugging in our choices of $\rho_0,\rho_1,\beta$, 
after simplifying we have
\[\dH^2\bigl(\psi^*(P),\psi^*(Q)\bigr)\geq c''_d \cdot \frac{2^d}{3^d} \cdot \sqrt{ \min\left\{\frac{s_d\delta}{2\eps},\frac{1}{4}\right\}}.\]
This completes the proof of the theorem, when $c_d$ is chosen appropriately.

\subsubsection{Completing the proof of Theorem~\ref{thm:lowerbd_empirical}}
The first term in the lower bound, i.e., $\sup_{P\in\Pcal_d:\Ep{P}{\norm{X}^q}\leq 1,\, \eps_P\geq \eps^*_d} \EE{\dH^2\bigl(\psi^*(\Pemp),\psi^*(P)\bigr)}\geq c_d n^{-\frac{2}{d+1}}$, 
holds by \citet[Theorem 1]{kim2016global}, which establishes this as the minimax rate (for $d\geq 2$) over distributions $P$ that are log-concave
(we can verify that the distribution $P$ constructed in their proof satisfies the conditions $\Ep{P}{\norm{X}^q}\leq 1$, $\eps_P\geq \eps^*_d$, for appropriately chosen $\eps^*_d$).  If instead $d=1$, then the first term cannot be the minimum.

Next, to prove the second term in the lower bound, we consider a mixture model. Let $P$ be the distribution constructed in~\eqref{eqn:lowerbd_construction_distribution} with 
\[\rho_0 = \frac{1}{2}, \ \rho_1 = n^{1/q}, \ \beta = \frac{1}{2n}.\]
Then clearly $\Ep{P}{\norm{X}^q}\leq 1$, and $\eps_P\geq \frac{1}{2}s_d$, so $\eps_P\geq \eps^*_d$ for an appropriately chosen $\eps^*_d$.
Now, with probability at least $1/2$, the observations $X_1,\dots,X_n$ are all drawn from the first component of the mixture model, i.e., $\psi^*(\Pemp)$
is supported on $\mathbb{B}_d(1/2)$. On this event, applying~\eqref{eqn:lowerbd_construction_dH_step} and plugging in our choices of $\rho_0,\rho_1,\beta$, 
after simplifying we have
\[\dH^2\bigl(\psi^*(\Pemp),\psi^*(P)\bigr) \geq c'''_d \cdot n^{-\frac{1}{2}+\frac{1}{2q}},\]
where $c'''_d$ depends only on $d$. This establishes the second term in the lower bound claimed in Theorem~\ref{thm:lowerbd_empirical}, and 
thus completes the proof of the theorem.

\section*{Acknowledgements}
The authors thank the anonymous reviewers and Oliver Feng for helpful comments.  R.F.B.~was 
supported by the National Science Foundation via grant DMS--1654076
and by an Alfred P.~Sloan fellowship. R.J.S.~was 
 supported by EPSRC grants EP/P031447/1 and EP/N031938/1.

\bibliographystyle{plainnat}
\bibliography{bib}

\end{document}